\documentclass[twoside, pagesize]{scrartcl} %
\usepackage{float} 
\usepackage{graphicx}
\usepackage{enumerate}
\usepackage{amssymb, amsmath, amsthm}
\usepackage{dsfont}
\usepackage{textcomp}
\usepackage{listings}
\usepackage{color} 
\usepackage{hyperref}
\usepackage[ngerman, english]{babel}
\usepackage[a4paper, left=3.5cm, right=3.5cm, top=2cm]{geometry} 

\newtheorem{proposition}{Proposition}[section] 
\newtheorem{defn}[proposition]{Definition}
\newtheorem{assumption}[proposition]{Assumption}
\newtheorem{thm}[proposition]{Theorem}
\newtheorem{lemma}[proposition]{Lemma}
\newtheorem{corollary}[proposition]{Corollary}
\newtheorem{remark}[proposition]{Remark}

\usepackage[T1]{fontenc}
\usepackage{lmodern}

\newcommand{\N}{\ensuremath{{\mathbb N}}}

\newcommand{\R}{\ensuremath{{\mathbb R}}}

\newcommand{\hta}{\ensuremath{{\hat \tau}}}

\newcommand{\E}{\ensuremath{{\mathbb E}}}

\newcommand{\Pro}{\ensuremath{{\,\mathbb P}}}

\def\SpecialChap#1{{\let\cleardoublepage\relax\chapter{#1}}}

\begin{document}
	\title{A Solution Technique for L\'evy Driven Long Term Average Impulse Control Problems}
\author{S\"oren Christensen\thanks{Christian-Albrechts-Universit\"at zu Kiel, Department of Mathematics, Ludewig-Meyn-Str. 4,
24098 Kiel, Germany, \emph{lastname}@math.uni-kiel.de} \and Tobias Sohr\footnotemark[1]}
\maketitle

\begin{abstract}
 This article treats long term average impulse control problems with running costs in the case that the underlying process is a L\'evy process. 
Assuming a maximum representation for the payoff function, we give easy to verify conditions for the control problem to have an $\left(s,S\right)$ strategy as an optimizer. The occurring thresholds are given by the roots of an explicit auxiliary function. This leads to a step by step solution technique whose utility we demonstrate by solving a variety of examples of impulse control problems. 
\end{abstract}

\section{Introduction}
Stochastic control techniques have proven themselves a useful tool whenever consecutive optimal decisions are to be made under uncertainty. Naturally this leads to a broad range of applications from finance and economics to the management of natural resources. The most common type of stochastic control problems, continuous control problems, however, suffer from one major drawback. Often the optimal strategy requires an infinite number of actions in a finite time interval and is therefore not realizable in practice. 
Instead modelling the underlying problem as an impulse control problem circumvents this issue by just allowing strategies that consist of countably many actions. Therefore, impulse control models are the natural choice when the underlying problems entail some fixed costs for each action or one aims for realizable optimizer. \\
Impulse control problems were intensively studied over decades. The foundation of the theoretical framework was laid in \cite{BLi} by connecting impulse control problems to QVIs. Due to the overwhelmingly broad field of applications and literature we restrain ourselves to present just a very selective digest of applications: Finance (\cite{Belak2019}, \cite{Korn1999}), control of the exchange rate (\cite{MundacaOeksendal1998}), optimal harvesting (\cite{A04}) and inventory control (\cite{helmes2017continuous}) are some application closely connected to this work. Furthermore, \cite{OeksendalSulem2005} provides a broad range of applications. While especially in financial settings most of the times a discounted payoff functional is used, in forest management and also in inventory control a long term average payoff or cost functional also is of interest, when one aims for a sustainable and long-ranging nature of the problem.
In discounted problems the value function can be characterized as the value of an implicit stopping problem and under quite general assumptions this yields a characterization of the value function as the smallest function in a certain set of superharmonic majorants of the payoff function (see \cite{C13impulse}). In contrast to that the value of long term average problems is often constant and therefore bears significantly less structure to work with. An additional mathematical difficulty is the lack of existence of a 0-resolvent, since in discounted problems resolvents or discounted potentials play a significant role in many solution techniques. Nevertheless, a connection to a stopping problem can be made, see \cite{stettner1986continuous} and \cite{stettnerpalczewski17}, but apart from results with strong ergodicity assumptions posed upon the process, or results for diffusion processes (\cite{jack2006impulse}, \cite{helmes2017weak}), there are very few works, especially when it comes to (semi-)explicit characterizations of value and optimal strategies. In the discounted setting, one such example in the branch of inventory control is \cite{yamazaki}, where the underlying process is assumed to be a spectrally one sided L\'evy process and the payoff function is assumed to be linear. \\
 Here in our work we study the long term average problem with generalized linear costs for an underlying L\'evy process. Under fairly general conditions we first give a verification theorem that utilizes martingale techniques and can be viewed as the long term average equivalent to comparable results for the discounted case in \cite{C13impulse}. Furthermore, we characterize the value of the impulse control problem as the value of a connected stopping problem and even construct an optimal strategy from the optimizer of said stopping strategy. We assume existence of a maximum representation of the payoff function, an idea that is also used in the discounted setting, see e.g. \cite{follmerknispel} for a detailed discussion of such representation and the connection to potential theory, and \cite{christensen2017impulse} for an application to impulse control. Using this representation, we obtain easy to verify sufficient conditions for a so called $(s,S)$ strategy (shifting the process down to $s$ whenever it exceeds $S$) to be optimal and develop a solution technique to semiexplicitly obtain these boundaries and the value in many cases of interest. With this solution technique we show existence of optimal $(s,S)$ strategies for the long term average version of many examples of interest of impulse control problems that are presented in \cite{OeksendalSulem2005}. As an application we take a closer look to inventory control problems for spectrally positive L\'evy processes. In this case we semiexplicitly characterize the value and the optimal boundaries by a maximization problem that only entails the L\'evy exponent of the underlying process and the root of its right inverse, similar to the characterization \cite{helmes2017continuous} for diffusions and the one \cite{yamazaki} found for spectrally one sided L\'evy processes with discounted payoff functional. However, our assumptions are less restrictive and our characterization seems to be a bit easier to apply. 
 
 \subsection{Structure of the article}
Before presenting the detailed notations and assumptions, in the second section we condense our findings to an explicit step by step solution technique and illustrate the applicability of our technique by showing existence of optimal $\left(s,S\right)$ strategies in quite general setting that occur e.g. in forest management and inventory control. 
 In Section \ref{sec:setup} we introduce the necessary notations for the following proofs, define and motivate the problem and collect a bunch of necessary general results as well as some first insights on the structure of the problem that will be needed later. 
 In Section \ref{sec:connection} we first prove a verification theorem and discuss the connection to a related auxiliary optimal stopping problem. 
%
 Section \ref{sec:OSP} is devoted to said stopping problem. Assuming existence of a maximum representation of the value function, we give sufficient conditions for the stopping problem to have a threshold time as an optimizer and we characterize this optimizer via the maximum representation. 
 In Section \ref{sec:restarting} we show that under the assumptions of section five an $\left(s,S\right)$ strategy is optimal for the control problem where $S$ is the threshold of the stopping problem's optimizer. In the case the ladder height process of the underlying process is a special subordinator (and not a compound Poisson process) we also characterize the value $s$ in terms of the maximum representation. 
 A way to obtain the aforementioned maximum representation by use of the ladder processes is developed in Section \ref{sec:discussion_ass}.
 Section \ref{sec:proof_sol_tech} connects the results of the first sections to a proof of the validity of section two's solution technique. 
 In Section \ref{sec:appl} we apply our solution technique onto inventory control and optimal harvesting.
  Here we are able to (semi-)explicitly calculate the optimal strategy and the value of problems in common examples from these fields.

 \section{Main Results}\label{sec:main}
 The core of our results is a step by step solution technique for long term average impulse control problems of the form \begin{align*}v=\sup_{S=(\tau_n,\xi_n)_{n\in \N}}\liminf_{T\rightarrow\infty} \frac{1}{T}\E_x\left(\sum_{n:\tau_n\leq T}\left(\gamma\left(X^S_{\tau_n,-}\right)-\gamma\left(\xi_n\right)-K\right)-\int\limits_0^{T}h\left(X^S_s\right)ds\right)
 \end{align*} where $X$ is a L\'evy process with $\E\left(X_1\right)>0$, $S=(\tau_n,\xi_n)_{n\in \N}$ are admissible control strategies consisting of a sequence of stopping times $(\tau_n)_{n\in \N}$ indicating when the process is shifted and $\mathcal F_{\tau_n}$ measurable random variables $\xi_n$ indicating whereto the process is shifted at time $\tau_n$, here arbitrary downshifts are allowed. $X^S$ denotes the controlled process. All the objects are thoroughly defined in Section \ref{notationen} -- here we want to state our main findings as briefly and quickly as possible. The pay-off function $\gamma$ is assumed to be non-decreasing and the running cost function $h$ is assumed to be non-negative and to fulfil some integrability conditions, the detailed assumptions can be found in Section \ref{bedingungstetig}.  
 The approach not only provides a sufficient criterion to verify existence of an optimal $\left(s,S\right)$ strategy, in many cases it leads to a semiexplicit or explicit characterization of the boundaries. Our step by step solution technique reads as follows: 
 \begin{enumerate}
 	\item Find a function $f$ such that for all $x,\bar y\in \R$ with $x<\bar y$ $$
 	\gamma\left(x\right)=\E_{x}\left[\int_{0}^{\tau_{\bar y}}f\left(\sup_{ r\leq t} X_r\right)dt\right]+\E_{x}\left[\gamma\left(X_{\tau_{{\bar y}}}\right)-\int_{{ 0}}^{\tau_{\bar y}} h\left(X_s\right) ds\right], $$ where $\tau_y:=\inf\{t\geq 0|X_t\geq y\}$ for all $y \in \R$. 
 	 One way to obtain such a function, provided the occurring objects exist, is the choice \begin{align*}
 	f=- \left(A_H \gamma + \hat h \right) 
 	\end{align*}
 	as discussed in Definition \ref{fdef} and thereafter. Here $A_H$ is the extended generator of the ascending ladder height process $H$ of $X$ (normed appropriately, see Definition \ref{ladderheightdefinition}) and for all $y \in \R$ $$
 	\hat h\left(y\right)=\E_y\left(\int\limits_0^\infty h\left( H^{\downarrow}_t\right) \ dt \right) =\int_{{ 0}}^\infty h\left(y+x\right) d U^\downarrow\left(dx\right)
 	$$ where $H^{\downarrow}$ is the descending ladder height process of $X$ and $U^\downarrow$ the occupation measure of $H^\downarrow$. 
 	\item Find $\rho^* \in \R$ such that $f\left(x\right)=\rho^* $ has exactly two solutions $\underline x_{\rho^*} < \overline x_{\rho^*}$ and \begin{align*}
 	0&=\sup_{x\in [\underline x_{\rho^*}, \overline x_{\rho^*}]}\E_x\left(\gamma\left(X_{\tau_{\overline x_{\rho^*}}}\right)-\int_{{ 0}}^{\tau_{\overline x_{\rho^*}}}\left(h\left(X_s\right)+ \rho^* \right)\ ds\right)
 	\\&= \sup_{x\in [\underline x_{\rho^*}, \overline x_{\rho^*}]}\E_x\left(\int_0^{\hat \tau_{\overline x_\rho}}f\left(\sup_{ r\leq t }X_r\right)- \rho^*\ dt \right).
 	\end{align*}
 	If such $\rho^*$ exists, we have $$v=\rho^*,$$ and the $\left(s,S\right)$ strategy with $$S=\overline x_{\rho^*}$$ and \begin{align*}
 	s:&=\arg\max_{x\in [\underline x_{\rho^*}, \overline x_{\rho^*}]}\E_x\left(\gamma\left(X_{\tau_{\overline x_{\rho^*}}}\right)-\int_{{ 0}}^{\tau_{\overline x_{\rho^*}}}\left(h\left(X_s\right)+ \rho^*\right)  \ ds\right)
 	\end{align*} is optimal. 
 	\item If $X$ is not a compound Poisson process and $H$ is a special subordinator (as defined in Definition \ref{specialsub}), we furthermore have $$s=\underline x_{\rho^*}.$$ 
 \end{enumerate}
 There are two things we want to remark: First, our results exceed this condensed solution technique and for example deliver $\epsilon$-optimal $\left(s,S\right)$ strategies, even if no optimizers exist. \\ 
 Second, the ascending and descending ladder height processes are in general difficult to handle. However, in many special cases there are numerous helpful results known about these processes that often enable us to handle the objects occurring in our solution technique quite well. Kyprianou's book (\cite{kypri}) is an excellent source for theoretical results about these processes. Also for a dense class of L\'evy processes, namely the ones whose positive jumps are of phase type distribution, in \cite{pistorius1} an iterative method to explicitly calculate the law of the ascending ladder height process was developed. Furthermore, for many examples of interest, not the full distribution of the ladder processes needs to be known, but only certain moments. 
 To illustrate the utility of our solution technique, in the following we look at some interesting examples and special cases, that cover the long term average equivalents to almost all examples of impulse control problems in \cite{OeksendalSulem2005}. We don't pose any restrictions on the process $X$ apart from sufficient integrability to make the problem non-degenerate.
 \subsection{Linear $\gamma$ and convex $h$}
 The first special case we want to focus on is the case when for all $x \in \R$ 
 
 \begin{align*}
 \gamma\left(x\right)=Cx 
 \end{align*}
 and $h$ is positive, convex, continuous and $\lim_{x\rightarrow \infty}h\left(x\right)=\lim_{x\rightarrow -\infty}h\left(x\right)=\infty. $
 This type of payoff and cost functions occur in inventory control, see the discussion in Section  \ref{subsec:inventory}.
%
 Also the long term average equivalent to the dividend problem presented in \cite{OeksendalSulem2005}, Example 6.4, and the prominent exchange rate control problem (Example 6.5 in \cite{OeksendalSulem2005}) is covered by this special case. \\
 With our solution technique  one can show existence of an optimal $\left(s,S\right)$ strategy quite easily.
 \begin{enumerate}
 	\item To obtain $f$ we first observe that $H$ is a subordinator and therefore $$A_H \gamma=C\delta+C\int_{{ 0}}^\infty y\Pi_H\left(dy\right)$$ for some constant $\delta$ and the jump measure $\Pi_H$ of $H$ hence $A_H \gamma(\cdot)$ is constant. Further, for all $y \in \R$ $$\hat h\left(y\right)=\int_{{ 0}}^\infty h\left(y+x\right) d U^\downarrow\left(dx\right),$$
 	hence $\hat h$ is still convex with $\lim_{x\rightarrow \infty}\hat h\left(x\right)=\lim_{x\rightarrow -\infty}\hat h\left(x\right)=\infty $ and hence the equation $f(x)=\rho$ always has exactly two solutions, if we choose $\rho$ large enough. 
 	\item The function given by \begin{align*}
 	\rho \mapsto \sup_{x\in [\underline x_{\rho}, \overline x_{\rho}]}\E_x\left(\int_0^{\hat \tau_{\overline x_{\rho}}}f\left(\sup_{ r\leq t }X_r\right)-\rho\ dt \right)
 	\end{align*} is monotone and continuous, hence either the intermediate value theorem provides the desired root $\rho^*$, or we are in a degenerate case and the value is either $\infty$ or $-\infty$. 
 \end{enumerate}
 Hence our solution technique verifies the existence of an $\left(s,S\right)$ strategy. \\
 Furthermore, concerning the explicit obtainability we want to remark that the extended generator of the ladder height process in general is difficult to obtain. But to obtain the function \begin{align*}
 f\left(x\right)&=-\left(A_H\gamma \left(x\right)+\hat h \left(x\right)\right)
 \\&=\left\{C\delta+C\int_{{ 0}}^\infty y\Pi_H\left(dy\right)+\E_x\left(\int\limits_0^\infty h\left( H^{\downarrow}_t\right) \ dt \right)\right\}
 \end{align*}
 one does not need full knowledge of $A_H$. Only the drift term $\delta$ and the expected jump size of the ladder height process are needed. These parameters in principal can be expressed in parameters of $X$ which in some cases leads to good characterizations, as we will see later. But more importantly they seem to be accessible by path-wise simulation techniques, e.g. Monte Carlo methods, since by simulating paths of the initial process $X$, one can often directly derive the ladder height processes' path and therefore also its jumps and drift parts. The same arguments hold for finding $\hat h$. In most treated examples, $h$ are relatively simple functions, like (piecewise) linear ones, the square function, or even just a constant function. Hence in that cases $\hat h$ is just an integral moment of the descending ladder height process. 
 \subsection{Polynomial $\gamma$ and $h$} 
To show that basically the same arguments for finding $f$ hold in more general cases we now turn our attention to the case where $\gamma$ and $h$ are polynomials. This example is certainly one of great interest, since on one hand polynomials are interesting special cases of payoff and running cost functions on its own. On the other hand polynomials may serve as a tool for approximating more general functions. In the following, we will see that the necessary transformations we have to apply on $\gamma$ and $h$ have the very compelling property to transform polynomials to polynomials of the same degree. This makes our solution technique boil down to an analysis of a polynomial of known degree whose coefficients can be expressed in terms of $\gamma$, $h$, and parameters of the process. \\ 
This setting includes the long term average analogon to Example 7.8 in \cite{OeksendalSulem2005}. 
\\
 So now we assume $
 \gamma\left(x\right)= \sum_{i=0}^l a_i x^i
 $
 and 
 $
 h\left(x\right)= \sum_{i=0}^k c_i x^i
 $ and for the sake of simplicity and brevity simply assume the occurring moments and integrals to exist. 
 Then \begin{align*}
 A_H \gamma \left(x\right)=\sum_{i=0}^l b_ix^i
 \end{align*}
 with \begin{align*}
 b_i=(i+1) a_{i+1}+\sum_{j=i}^{l} a_j\binom{j}{i}\int_{{ 0}}^\infty y^{j-i}d\Pi_H\left(y\right) 
 \end{align*}
 for all $i \in \{0, ...,n\}$ where $a_{l+1}=0$ and $\Pi_H$ is the L\'evy measure of $H$. Furthermore, using Fubini's theorem \begin{align*}
 \hat h\left(x\right)
= \sum_{j=0}^k d_j x^j 
 \end{align*}
 where \begin{align*}
 d_j:&=\sum_{i=1}^{n}c_i \binom{i}{j}\int_{{ 0}}^\infty\E\left(\left(H_t^\downarrow\right)^{i-j} \right) \ dt
 \end{align*}
 Again we remark that useful formulas to obtain $\Pi_H$ can be found in \cite{kypri}. Further, $H^\downarrow$ acts in law like an exponentially killed subordinator, hence the occurring moments can be obtained via the cumulant function of said subordinator. Also we don't need to know the full distribution of ascending and descending ladder height processes. To explicitly get the occurring coefficients one only has to calculate drift rate, moments of the jump measure of the ascending ladder height process, and cumulative moments of the descending ladder height process, both up to a previously known given degree. So again this is a good starting point for simulations.\\
 The observation that $f=-\left(A_H\gamma+\hat h\right)$ is also a polynomial with degree $\max\{n,k\}$ on top of that yields a starting point for a different procedure to (computationally) find $f$. Our calculations above show that $f$ is a polynomial of known degree. Hence it is possible to just start with the desired equation
 $$
 \gamma\left(x\right)=\E_{x}\left[\int_{0}^{\tau_{\bar y}}f\left(\sup_{ r\leq t} X_r\right)dt\right]+\E_{x}\left[\gamma\left(X_{\tau_{{\bar y}}}\right)-\int_{{ 0}}^{\tau_{\bar y}} h\left(X_s\right) ds\right], $$ plug in a general polynomial of the right degree for $f$, compute the occurring moments and integrals either explicitly or approximate them numerically and then compare the coefficients. For example the aforementioned work \cite{pistorius1} also gives an iterative method to explicitly obtain the Laplace transform of the law of the running supremum of $X$, in case the upward jumps of $X$ are of phase type. 
 
 \subsection{Exponential L\'evy processes}
 A class of processes which is of great interest in mathematical finance is the class of exponential L\'evy processes, for example Exercises 6.2 and 7.2 in \cite{OeksendalSulem2005} present examples where this setting is used to determine the optimal stream of dividends.\\
 Set for all $x \in \R$ $$\gamma\left(x\right):= e^x $$ and for all $x \in \R$ 
 $$h\left(x\right):= a_1e^{a_2 x}+b_1e^{-b_2 x}.$$
 Again apart from $\E\left(X_1\right)>0$, we only assume on $X$ that the occurring moments and integrals to exist. \\
 To obtain $f$ similar to the polynomial case we see that \begin{align*}
 A_H \gamma\left(x\right)&=\delta e^x+e^x\int_{{ 0}}^\infty e^y\Pi_H\left(dy\right)\\
 &=\delta e^x+e^x\int_{{ 0}}^\infty e^y\int_{{ 0}}^\infty \Pi\left(z+y,\infty\right) U^\downarrow\left(dz\right)\  dy.\\
 \end{align*} Further, \begin{align*}
 \hat h\left(x\right)&=\int_{{ 0}}^\infty a_1e^{a_2 \left(x+y\right)}-b_1e^{b_2 \left(x+y\right)} d U^\downarrow\left(dy\right) \\&=a_1e^{a_2x}\int_{{ 0}}^\infty e^{a_2 y} d U^\downarrow\left(dy\right)+ b_1e^{-b_2x}\int_{{ 0}}^\infty e^{-b_2 y} d U^\downarrow\left(dy\right)
 \end{align*}for all $x\in \R$. Hence in this case finding the optimal value and an optimal strategy boils down to the analysis of exponential functions.

\section{Setup, toolbox and general results}\label{sec:setup}
 Having presented a bunch of examples that show various use cases of our solution technique, the rest of this work will be devoted to prove that the solution technique indeed works as described. First, we formally state the problem, give detailed definitions and introduce the necessary notations. Then we connect the impulse control problem to an associated stopping problem and describe, when the stopping problem has an optimal threshold time by use of the maximum representation. After that we discuss when certain roots of $f$, the function that occurs in said maximum representation, yield boundaries for an optimal $\left(s,S\right)$ strategy. Lastly, we describe one possible way to obtain said function $f$. 
\subsection{Notation and prerequisites}\label{notationen}

Let $X$ be a L\'evy process on $\R$ with a right continuous filtration $\mathcal F:=\left(\mathcal F_t\right)_{t\geq 0}$ such that $\mathcal F_0$ is complete. Denote the underlying probability space with $\left(\Omega, \mathcal P, \Sigma\right)$ and for all $x \in \R$ define $\Pro_x\left(\cdot\right):=\mathcal P\left(\cdot|X_0=x\right)$. For each $x \in \R$ let $\E_x$ be the expectation operator associated to $\Pro_x$ and shortly write $\Pro:=\Pro_0$ and $\E:=\E_0$. Further, we assume existence of a timeshift operator $\theta$. Through the whole article we assume that $\E(X_1)$ exists and $0<\E(X_1)<\infty$.\\
\subsubsection{Generalities about L\'evy processes}
Before we formally define control problems, we briefly collect the needed results and notations about L\'evy processes. 
\begin{defn}
	For all $x \in \R$ set $$\tau_x:=\inf\{t\geq 0 \mid X_t\geq x\}$$ and $$\mathring\tau_x:=\inf\{t\geq 0 \mid X_t> x\}.$$
\end{defn}
\begin{lemma}[Wald's equation, continuous version]\label{continuouswald}
	Let $Y$ be a L\'evy process such that $\E(Y_1)$ exists and $0<\E(Y_1)\leq\infty$. Let $\tau$ be a stopping time. Then, $$\E \left(X_\tau\right) = \E\left(Y_1\right)\E\left(\tau\right). $$
\end{lemma}
\begin{proof}
	This result supposedly goes back to Doob in 1957 and an even more general version can be found in \cite[Corollary 1]{waldsequation}. 
\end{proof}
\begin{lemma}\label{firstentryfinitenessLevyLEM}
	Let $Y$ be a L\'evy process such that $\E(Y_1)$ exists and \linebreak $0<\E(Y_1)<\infty$. Then for all $a\geq 0$ holds $\E(\tau_a)<\infty$ and $\E(Y_{\tau_a})<\infty$. 
\end{lemma} 
\begin{proof}
	The first part of the claim is a direct consequence from the analogous result for random walks that is proven in \cite[Theorem 2.1]{gut74}. The second part then follows with Lemma \ref{continuouswald}.
\end{proof}

\begin{lemma}\label{tauoffen=tauabgeschl levy}
	 If $X$ is not a compound Poisson process or has a L\'evy measure with no atoms, then for all $x \in \R$ holds $\mathring{\tau}_x=\tau_x$ a.s. under all $\Pro_y$, $y \in \R \setminus \{x\}$. 
\end{lemma}
\begin{proof}
	This is proven in \cite[Lemma 2]{pecherskiirogozin} in case that $X$ is not a compound Poisson process. The case that $X$ is a compound Poisson process whose L\'evy measure has no atoms follows with elementary arguments. 
\end{proof}

\begin{lemma}\label{xicontinuousLevyLEM} Assume $X$ is not a compound Poisson process or has a L\'evy measure with no atoms. 
	Define the mapping $\xi$ by $$
	\xi(x,y):=\E_x(\tau_y) $$
	for all $x,y\in \R$ with $x< y$. 
	Then, $\xi$ is a continuous real valued mapping, that is non-decreasing in the second and non-increasing in the first argument.
\end{lemma}
\begin{proof}
	Lemma \ref{firstentryfinitenessLevyLEM} yields that for all $x,y\in \R$ with $x\leq y$ holds $\E_x(X_{\tau_y})<\infty$. Further, Lemma \ref{tauoffen=tauabgeschl levy} implies that for all $y \in \R$ holds $\lim_{a\nearrow y} \tau_a = \tau_y$ a.s. under all $\Pro_z$, $z\in \R\setminus \{y\}$. With dominated convergence we hence get continuity of $\xi$ in the second argument, but for all $x,y\in \R$ with $x< y$ we can use the homogeneity of $X$ to write $\E_x(\tau_y)=\E(\tau_{y-x})$, which yields the claim. 
\end{proof}
\begin{lemma}\label{XIcontLEMlevy}
	Let	$f:\E\rightarrow \R$ be a continuous function, $\overline x \in \R$ a root of $f$ and assume that $X$ is a subordinator that is either no compound Poisson process or a compound Poisson process whose L\'evy measure has no atoms. Then the function \begin{align*}
	\Xi:\R \rightarrow \R; x\mapsto \E_x\left(\int\limits_0^{\tau_{\overline x}}f(X_s) ds \right)
	\end{align*} is continuous. 
\end{lemma}
\begin{proof} 
	Continuity in all points $y\in \R$ with $y<\underline x$ follows from approximation with simple functions and Lemma \ref{xicontinuousLevyLEM}. To see that $\Xi$ is continuous on whole $\R$, note that $\Xi(y)=0$ for all $y\in \R$ with $y\geq \overline x$ hence it remains to be shown that $\lim\limits_{y\nearrow \overline x}\Xi(y)=0$. Let $z\in \R$ with $z<\overline x$. Let $\epsilon >0$. Then for all $y\in \R$ with $\overline x-y<\overline x-z$ such that $ {\max_{x\in [y,\overline x]}|f(x)|}<{\E_{z}(\tau_{\overline x})} \epsilon$ we have $$\Xi(y)<\epsilon$$ and because $f$ is continuous and $\overline x$ a root of $f$, the set of such values $y$ is a non-emty interval. 
\end{proof}
\label{ladderheightprocess}
Following \cite[Definition 6.1]{kypri} we define a local time at the maximum of $X$ as a continuous, non-decreasing, adapted process $(L)_{t\in [0,\infty)}$ on $[0,\infty)$ with the following properties: \begin{enumerate}
	\item The support of $dL$ is $\overline{\{t\in [0,\infty)  \mid \overline X_t=X_t\}}$.
	\item For each stopping time $\tau$ with the property, that a.s. $X_\tau=\overline X_\tau$, the process $$(L_{\tau+t}-L_\tau)_{t\in [0,\infty)}$$ is independent of $\mathcal F_\tau$ and is distributed as $(L_{t})_{t\in [0,\infty)}$ under $\Pro$.
\end{enumerate}
Such a local time exists for a wide class of L\'evy processes. Nevertheless, if (and only if) $0$ is not regular for the positive halfline (this means that $\mathring{\tau}_0\neq 0$ a.s. under $\Pro_0$) such a continuous local time fails to exist, then it is possible to construct a right continuous alternative we will tacitly work with instead, see \cite[Theorem 6.6]{kypri}, the references thereafter for the proof of existence and \cite[Section 6.1]{kypri} for more details and explicit constructions of local times in several cases. For a local time $L$ we set $L_\infty:=\lim\limits_{t\rightarrow \infty} L_t$ and define the inverse local time process $L^{-1}$ by \begin{align*}
L^{-1}_t:=\begin{cases}
\inf\{s>0  \mid L_s>t\}; &\text{  if } t <L_\infty\\
\infty; &\text{  else}
\end{cases} \ \ \ \ \ \ \ &\forall t\in[0,\infty)
\end{align*} It can be seen in the definition that a local time can only be unique up to a multiplicative factor, which we chose conveniently for our purposes in the next definition. 
\begin{defn}\label{ladderheightdefinition}
	Let $L$ be a local time at the maximum and $H$ defined by $$H_t:=X_{L^{-1}\left(t\right)}$$ for all $t \geq 0$ the ladder height process. $\left(H,L^{-1}\right)$ is a L\'evy process, even a bivariate subordinator. With Wald's equation (see Lemma \ref{continuouswald}), we have $\E\left(L^{-1}\left(\tau_x\right)\right)= \E\left(L_1^{-1}\right)^{-1}\E\left(\tau_x\right).$ Since $L$ is only defined up to a multiplicative constant, w.l.o.g. we choose $L$ such that $\E\left(L_1^{-1}\right)=1$ and, hence, $\E\left(L^{-1}\left(\tau_x\right)\right)=\E\left(\tau_x\right)$ for all $x \in \R$.   
	Further, we set $$\hat{\tau}_x:=L^{-1}\left(\tau_x\right)=\inf\{t\geq 0  \mid H_t\geq x\}$$ for all $x \in \R$. \\
	In the same way we define the descending ladder height process $H^\downarrow$ as the ladder height process of $-X$.\\
\end{defn}

\subsubsection{Impulse control problems}
Let $\gamma$ and $h$ be real functions and assume:
\begin{assumption}\label{bedingungstetig}
	\begin{enumerate}
		\item $\gamma$ is nondecreasing and differentiable. 
		\item\label{hassumption} $h$ is nonnegative, continuous and for all $x,y \in \R$ with $x<y$, we have $$\E_x\left(\int_0^{\tau_y}h\left(X_s\right)\ ds\right)< \infty.$$ 
	\end{enumerate}
\end{assumption}
Define the set $\mathcal T$ as the set of all stopping times $\tau$ such that we have $\E_x\left(\tau \right)<\infty$ and $ \E_x\left(\int_0^\tau h\left(X_s\right)~ds\right)<\infty$ for all $x \in \R$.
A control strategy $S=\left(\tau_n,\xi_n\right)_{n \in \N}$ consists of a sequence of stopping times $\left(\tau_n\right)_{n \in \N}$ in $\mathcal T$ that fulfils $\lim_{n\rightarrow \infty}~\tau_n=\infty$ a.s. under all $\Pro_{x}$ and $\mathcal F_{\tau_n}$ measurable random variables $\xi_n$ with $\E_y(\gamma(\xi_n))\in \R$ indicating whereto the process is shifted at time $\tau_n$. We model the controlled process recursively by $$
X^S_t:=X_t-\sum_{n;\tau_n \leq t}\left(X^S_{\tau_n,-}-\xi_n\right)$$ for each strategy $S=\left(\tau_n,\xi_n\right)_{n \in \N}$. Herein we use  $$X^S_{\tau_n,-}:=X_{\tau_n}-\sum_{i=1}^{n-1}\left(X^S_{\tau_i,-}-\xi_i\right)$$ for the value right before the $n$-th shift (Note that due to $X$ not being continuous this value may deviate from both $X^S_{\tau_n}$ and $X^S_{\tau_n-}$). 
We define the long term average value of the process controlled by an  strategy $S=\left(\tau_n,\xi_n\right)$ by \[
J_x\left(S\right):= \liminf_{T\rightarrow\infty} \frac{1}{T}\E_x\left(\sum_{n:\tau_n\leq T}\left(\gamma\left(X^S_{\tau_n,-}\right)-\gamma\left(\xi_n\right)-K\right)-\int\limits_0^{T}h\left(X^S_s\right)ds\right)
\]
where $\gamma$ is the so called payoff function, $K>0$ models fixed cost and $h$ is called running costs. 
Further, for each $B\subseteq \R$ we define $\mathcal S_B$ as the set of all control strategies $S=\left(\tau_n,\xi_n\right)_{n \in \N}$ such that $X^S_{\tau_n,-}\geq \xi_n\in B$  a.s. under all $\Pro_{x}$ for all $ n \in \N$ and call all elements of $\mathcal S_B$ admissible strategies. We fix $B \subseteq \R$ throughout the following sections and let  \begin{align}
v\left(x\right):=\sup_{S \in \mathcal{S}_{B} }J_x\left(S\right) \label{vdefEQ}
\end{align} define the value function for all $x \in \R$.\\
Before defining further necessary objects, let us make one remark on some of the assumptions. 
The Assumption \ref{bedingungstetig}, \ref{hassumption}. is a quite natural one to make. Without this condition, the state space is basically divided in two or more regions such that we cannot let the process go from one region to the other, otherwise we have to pay an infinite amount of costs. So in this case we would basically end up with several disjoint control problems, depending on the starting point. 
On the other hand, with the later developed tools and notations it will be clear that the Assumption \ref{bedingungstetig}, \ref{hassumption}. is not too restrictive and holds for almost all choices for $h$. So the finiteness of the integral depends mainly on the length and amplitude of excursions from the maximum of $X$ and since these are not dependent on the starting point, these integral will for most functions $h$ either for all or for no pair of points $x,y$ with $x<y$ be finite.\\

\subsubsection{Some important stopping times and strategies}\label{threshholdtimeandstrategydefinition}

Because threshold times occur frequently, we write $\tau_x=\inf\{t\geq 0|X_t\geq x\}$ for all $x \in \R$.
Since these strategies play an important role later on, we set $$\mathcal T_x:=\{\tau \in \mathcal T|X_\tau\geq x \ \text{a.s. under} \ \Pro_x \}$$ for each $x\in \R$ (where $\mathcal T$ is defined at the beginning of Section~\ref{notationen})  and for all $\tau \in \mathcal T_x$ we set $$\tau_1:=\tau,$$ $$\tau_n:=\tau\circ \theta_{\tau_{n-1}}+\tau_{n-1} \text{ for all } n>1$$ and set $R\left(\tau,x\right):=\left(\tau_n,x\right)_{n\in \N}$. Note that $R\left(\tau,x\right):=\left(\tau_n,x\right)_{n\in \N}$ is an admissible strategy whenever $\E_x\left(\tau\right)>0$ and $x\in B$. \\


\subsubsection{Degenerate case}
As last precaution, we take a look at the degenerate case that an infinite gain per period is possible. 
\begin{lemma}\label{endlichkeitslemma}
If there is $x \in B$ and $\tau \in \mathcal T_x$ such that $$\E_x\left[\gamma\left(X_{\tau}\right)-\int_{{ 0}}^\tau h\left(X_s\right) \ ds  \right]= \infty,$$ then $v\left(y\right)=\infty$ for all $y \in \R$. 
\end{lemma}
\begin{proof}
 We take the strategy $S:=\left(\tau_i,x\right)_{ i \in \N}:=R\left(\tau,x\right)$, write $\Delta \tau_i:=\tau_i-\tau_{i-1}$ for all $n \in \N$ and making use of Lemma \ref{rere} we first show that $J_x\left(S\right)=\infty$. To that end define for all $i \in \N$ $$R_i:=\gamma\left(X^S_{\tau_{i,-}}\right)-\gamma\left(x\right)-K.$$ Now $\left(\Delta \tau_i,R_i\right)_{i \in \N}$ is a sequence of i.i.d. random variables under $\Pro_{x}$, but the $R_i$ violate the integrability requirements of Lemma \ref{rere}. To circumvent that issue, we note that since for all $i \in \N$ $$R_i \stackrel{d}{=} \gamma\left(X_{\tau}\right)  -\gamma\left(x\right)-K,$$ under $\Pro_x$ for the negative part $R_1^-$ of $R_1$ we have $\E_x\left(R_1^-\right)<\infty$. Hence for all $a>0$ the random variable $R_i \wedge a$ is integrable, and  $\left(\Delta \tau_i,R_i\wedge a\right)_{i \in \N}$ fulfils the requirements of Lemma \ref{rere}. This yields 
 \[\frac{1}{T}\E_x\left(\sum_{i=1}^{N\left(T\right)}R_i \wedge a \right)\stackrel{T \rightarrow \infty}{\rightarrow}\frac{\E_x\left(R_1 \wedge a\right)}{\E_x\left( \tau\right)}
 \] where $N\left(t\right):=\sup\{n|\tau_n \leq t\}$. We tackle the running cost term also with Lemma \ref{rere}. To this end, we define $Q_i:=\int_{\tau_{i-1}}^{\tau_i}h\left(X^S_s\right) \ ds $ and note that $\left(\Delta \tau_i, Q_i \right)_{i \in \N}$ also fulfils the requirements of Lemma \ref{rere} and hence we get \begin{align*} \frac{1}{T}&\E_x\left(\int\limits_0^{T}h\left(X^S_s\right)ds\right)\\&\leq\frac{1}{T}\E_x\left(\int\limits_0^{\tau_{N\left(T\right)+1}}h\left(X^S_s\right)ds\right)
 \\&= \frac{1}{T}\E_x\left(\sum_{i=1}^{N\left(T\right)+1}Q_i\right)
 \\&\stackrel{ T \rightarrow \infty}{\rightarrow}\frac{\E_x\left(Q_1\right)}{\E_x\left(\tau\right)}
 \\&=\frac{\E_x\left(\int_{{ 0}}^\tau h\left(X_s\right) \ ds\right)}{\E_x\left(\tau\right)}\\&=:C.
 \end{align*}
 Now for each $T\geq 0$ and each $a \geq 0$: \begin{align*}
\frac{1}{T}&\E_x\left(\sum_{n:\tau_n\leq
	 T}\left(\gamma\left(X^S_{\tau_n,-}\right)-\gamma\left(x\right)-K\right)-\int\limits_0^{T}h\left(X^S_s\right)ds\right)
 \\\geq \frac{1}{T}&\E_x\left(\sum_{n=1}^{N\left(T\right)}\left( R_n\wedge a \right)-\int\limits_{0}^{T}h\left(X^S_s\right)ds\right)\\\geq & \lim_{T \rightarrow \infty}  \frac{\E_x\left(R_1\wedge a\right)}{\E_x\left(\tau\right)}+C.
 \end{align*}
 Now the monotone convergence theorem yields \[\lim_{a\rightarrow \infty} \E_x\left(R_1 \wedge a\right)=\E_x\left(R_1\right)=\infty\] and we finally get 
 $$ \frac{1}{T}\E_x\left(\sum_{n:\tau_n\leq
 	T}\left(\gamma\left(X^S_{\tau_n,-}\right)-\gamma\left(x\right)-K\right)-\int\limits_0^{T}h\left(X^S_s\right)ds\right) \stackrel{T \rightarrow \infty}{\rightarrow} \infty. $$
 It remains to show that also for all $y \in \R$ with $y \neq x$ holds $J_y\left(S\right)=\infty$. This can be easily done by adding a new first control to the strategy constructed above, where we shift the process back to $x$ as soon it exceeds $x$ for the first time. Assumption \ref{bedingungstetig}, 3. ensures that this is still an admissible control strategy and the renewal processes we worked with above then are delayed renewal processes, hence the renewal theoretic results we used still hold, as is worked out e.g. in \cite{asmussen2008applied}. 
\end{proof}



\section{Connection of the Control Problem to Martingales and Optimal Stopping}\label{sec:connection}
This section has two parts. The first part contains a verification theorem that connects the optimal value of the control problem and optimal strategies to a supermartingale. The second part that is also one of the main ingredients of the solution technique in section two connects the control problem to an associated optimal stopping problem. It is also shown that an optimal stopping time for said problem and an optimal starting point for the problem can be merged to an optimal control strategy. 
\subsection{Verification}
The main theorem of this section, Theorem~\ref{veridrei}, establishes sufficient conditions for a strategy to be an optimizer and a real number to be the value of the control problem that later on will help to verify optimality. This enables us to provide sufficient conditions under that strategies consisting of repeatedly stopping at the optimizer of the stopping problem and shifting back to a fixed point specified therein are optimal control strategies. 
\begin{thm}\label{veridrei}
	Let $g$ be a measurable function on $\R$, let $u$ be defined by $$u\left(x,y\right)=\gamma\left(x\right)-\gamma\left(y\right)-K-g\left(x\right)+g\left(y\right)$$ for all $x,y\in \R$ with $y\leq x$, let $\rho \in \R$ and define $$M:=\left({g\left(X_t\right)}-\int_{0}^{t}\left( h\left(X_s\right)+\rho\right) \ ds\right)_{t\geq0}$$
	and for each $S\in  \mathcal S_B$ set $$M^S:=\left({g\left(X^S_t\right)}-\int_{0}^{t}\left( h\left(X^S_s\right)+\rho\right) \ ds\right)_{t\geq0}.$$ 
	\begin{enumerate}[(i)]
		\item\label{item:a} Assume
		    \begin{enumerate}
			
			\item\label{item:a:a} $M$ is a supermartingale under $\Pro_{x}$ for all $x\in \R$,
			\item\label{item:a:b} \[\limsup_{T\rightarrow\infty}\frac{\E_{x}g\left(X^S_T\right)}{T}\geq 0
			\text{ for all } S\in\mathcal{S}_B, \ x \in \R, \]
			\item\label{item:a:c} \[u\left(x,y\right) \leq 0 \text{  for all } x\in \R, y \in B \text{ with } \ y \leq x . \]
		\end{enumerate} Then
		\[v\left(x\right)\leq \rho \text{ for all } x\in \R.\]
		
			\item\label{item:c} If there is a strategy $S^{\uparrow}=\left(\tau_n^{\uparrow},\xi_n^{\uparrow}\right)_{n \in \N}\in \mathcal S_B$  such that
		\begin{enumerate}
			\item\label{item:c:a} 
			\begin{equation*}\E_{x}\left(M_{\tau_n^{\uparrow} \wedge T,-}-M_{\tau_{n-1}^{\uparrow}\wedge T}\right)\geq 0 \text{ for all } n\in \N,\ x\in \R,\ T\geq 0,
			\end{equation*}
			\item\label{item:c:b} 
			\[\lim_{T\rightarrow\infty}\frac{\E_{x}g\left(X^{S^{\uparrow}}_T\right)}{T}\leq 0 \text{  for all } x\in \R, \]
			\item\label{item:c:c}  \[u\left(X_{\tau_n^{\uparrow},-},\xi_n^{\uparrow}\right)\geq 0  \ \ \Pro^{S^{\uparrow}}_{x} \text{-a.s. for all } x\in \R, \ n \in \N. \]
		\end{enumerate}
		Then \[ v\left(x\right)\geq J_x\left({S^\uparrow}\right)\geq \rho, \text{ for all } x \in \R.\]

		\item\label{item:b} If (\ref{item:a}) holds and a strategy $S^*=\left(\tau_n^*,\xi_n^*\right)_{n \in \N}\in \mathcal S_B$ as in \eqref{item:c} exists,
		then \[ v^{\mathcal S}\left(x\right)= \rho \ \text{ for all } x \in \R\]
		and $S^*$ is optimal in $\mathcal S_B$ in the sense that $v\left(x\right)= J_x\left({S^*}\right)$ for all $x\in \R$.
	\end{enumerate}
\end{thm}

\begin{proof}
	We first fix $S=\left(\tau_n,\xi_n\right)_{n\in \N}\in \mathcal S_B$ and $T>0$. Since the process $X$ runs uncontrolled on each stochastic interval $[\tau_{k-1},\tau_k)$, the optional sampling theorem yields that $\E_{x}\left(M^S_{\tau_k\wedge T-}-M^S_{\tau_{k-1}\wedge T}\right)\leq 0$ for each $k\in \N$, $x\in \R$. Hence
	\begin{align*}
& \E_x\left[\sum_{n \in \N :\tau_n\leq T}\left(\gamma\left(X^S_{\tau_n,-}\right)-\gamma\left(\xi_n\right)-K\right) -\int_{0}^{T} h\left(X^S_s\right)\ ds\right]
\\&\leq \E_x\left[\sum_{n \in \N:\tau_n\leq T}\left(\gamma\left(X^S_{\tau_n,-}\right)-\gamma\left(\xi_n\right)-K\right)\textcolor{white}{\int_{0}^{T} \int_{0}^{T} \int_{0}^{T} }\right.
\\ & \left.\textcolor{white}{\sum_{ T} \int_{0}^{T} }-\sum_{k=1}^\infty \left(M^S_{\tau_k\wedge T,-}-M^S_{\tau_{k-1}\wedge T}\right)-\int_{0}^{T} h\left(X^S_s\right)\ ds\right]\\
&=\E_x\left[\sum_{n \in \N:\tau_n\leq T}\left(\gamma\left(X^S_{\tau_n,-}\right)-\gamma\left(\xi_n\right)-K\right) -\sum_{k=1}^\infty \left(g\left(X^S_{\tau_k\wedge T,-}\right)-g\left(X^S_{\tau_{k-1}\wedge T}\right)\textcolor{white}{\sum_{T_T}\int_0^\tau}\right.\right.
 \\& \left.\left.\textcolor{white}{\sum_{ T} \int_{0}^{T} } -\int_{\tau_{k-1}\wedge T}^{\tau_k\wedge T} h\left(X^S_s\right)\ ds-\rho\left({\tau_k\wedge T}\right)+\rho\left({\tau_{k-1}\wedge T}\right) \right)-\int_{0}^{T} h\left(X^S_s\right)\ ds\right] \\
&=\E_x\left[\sum_{1\leq n:\tau_n\leq T}\left(\gamma\left(X^S_{\tau_n,-}\right)-\gamma\left(\xi_n\right)\right.\right.\\&\left.\left.\textcolor{white}{\sum_{n \in \N :\tau_n\leq T} \int_{0}^{T} }-K-g\left(X^S_{\tau_n,-}\right)+g\left(\xi_n\right)\right)-g\left(X^S_T\right)+g\left(X^S_0\right)+\rho T\right]\\
&=\E_x\left[\sum_{1\leq n:\tau_n\leq T}u\left(X^S_{\tau_n,-,\xi_n}\right)\right]-\E_xg\left(X^S_T\right)+g\left(x\right)+\rho T
\\&\leq - \E_xg\left(X^S_T\right)+g\left(x\right)+\rho T.
\end{align*}
	Dividing by $T$ and taking the limit $T\rightarrow\infty$, we obtain the first assertion.\\
	 To prove \eqref{item:c}, we see that similar calculations for $S^\uparrow$ yield
	
		\begin{align*}
	& \E_x\left[\sum_{n \in \N :\tau^\uparrow_n\leq T}\left(\gamma\left(X^{S^\uparrow}_{\tau^\uparrow_n,-}\right)-\gamma\left(\xi^\uparrow_n\right)-K\right) -\int_{0}^{T} h\left(X^{S^\uparrow}_s\right)\ ds \right]
	\\&\geq \E_x\left[\sum_{n \in \N:\tau^\uparrow_n\leq T}\left(\gamma\left(X^{S^\uparrow}_{\tau^\uparrow_n,-}\right)-\gamma\left(\xi^\uparrow_n\right)-K\right)\right. 
	\\& \left.\textcolor{white}{\sum_{ T} \int_{0}^{T} }-\sum_{k=1}^\infty \left(M^{S^\uparrow}_{\tau^\uparrow_k\wedge T,-}-M^{S^\uparrow}_{\tau^\uparrow_{k-1}\wedge T}\right)-\int_{0}^{T} h\left(X^{S^\uparrow}_s\right)\ ds\right]
	\\&=\E_x\left[\sum_{n \in \N:\tau^\uparrow_n\leq T}\left(\gamma\left(X^{S^\uparrow}_{\tau^\uparrow_n,-}\right)-\gamma\left(\xi^\uparrow_n\right)-K\right)\right. \\&\left.\left.\textcolor{white}{\sum_{T} \int_{0}^{T} }-\sum_{k=1}^\infty \left(g\left(X^{S^\uparrow}_{\tau^\uparrow_k\wedge T,-}\right)-g\left(X^{S^\uparrow}_{\tau^\uparrow_{k-1}\wedge T}\right)\textcolor{white}{\int^{T^T}_{T_T}} \right.\right.\right.
	\\&\left. \left. \textcolor{white}{\sum_{ T} \int_{0}^{T} } -\int_{\tau^\uparrow_{k-1}\wedge T}^{\tau^\uparrow_k\wedge T} h\left(X^{S^\uparrow}_s\right) \ ds-\rho{\tau^\uparrow_k\wedge T}+\rho\tau^\uparrow_{k-1}\wedge T\right)-\int_{0}^{T} h\left(X^{S^\uparrow}_s\right)\ ds\right] 
	\\&=\E_x\left[\sum_{1\leq n:\tau^\uparrow_n\leq T}\left(\gamma\left(X^{S^\uparrow}_{\tau^\uparrow_n,-}\right)-\gamma\left(\xi^\uparrow_n\right)\textcolor{white}{\sum_{ T} \int_{0}^{T} }\right.\right.
	\\&\left.\left.\textcolor{white}{\sum_{ T} \int_{0}^{T} }-K-g\left(X^{S^\uparrow}_{\tau^\uparrow_n,-}\right)+g\left(\xi^\uparrow_n\right)\right)-g\left(X^{S^\uparrow}_T\right)+g\left(X^{S^\uparrow}_0\right)+\rho T\right]
	\\&=\E_x\left[\sum_{1\leq n:\tau^\uparrow_n\leq T}u\left(X^{S^\uparrow}_{\tau^\uparrow_n,-,\xi^\uparrow_n}\right)\right]-\E_xg\left(X^{S^\uparrow}_T\right)+g\left(x\right)+\rho T
	\\&\geq - \E_xg\left(X^{S^\uparrow}_T\right)+g\left(x\right)+\rho T.
	\end{align*}
	Again $T\rightarrow \infty$ yields the claim. \\
	Lastly, \eqref{item:b} is a direct consequence of \eqref{item:a} and \eqref{item:c}. 
\end{proof}

	\subsection{Reduction to stopping}

In this section we characterize the value of the impulse control problem by the value of a stopping problem that resembles the maximal gain with one control. In that process we also show that optimal stopping times for that stopping problem, when repeatedly used, form an optimal control strategy. 
	 For all $y, x\in \R$ and all $\rho \in \R$ we set \begin{align*}
	 g^{\mathcal T_y}_\rho\left(x\right):=\sup_{\tau \in \mathcal T_y}\E_x\left(\gamma\left(X_\tau\right)-\int_{0}^{\tau}\left(h\left(X_t\right)+\rho\right) \ dt \right)
	 \end{align*}
	 and \begin{align*}
	  \mathfrak g^{\mathcal T_y}_\rho\left(x\right):=\sup_{\tau \in \mathcal T_y}\E_x\left(\gamma\left(X_\tau\right)-\gamma\left(x\right)-K-\int_{0}^{\tau}\left(h\left(X_t\right)+\rho\right) \ dt \right).
	 \end{align*}
Recall that herein $\mathcal T_y$ was defined as the set of all stopping times $\tau$ with $\E_y\left(\tau \right)<\infty$ and $ \E_y\left(\int_0^\tau h\left(X_s\right)~ds\right)~<~\infty$ and $X_\tau\geq y$ under $\Pro_y$. 
	\begin{remark}
		Looking at the definition of $\mathcal T$ and $\mathcal T_y$ in Section~\ref{notationen} and~\ref{threshholdtimeandstrategydefinition} we see that for all $y, x\in \R$ the expressions $g^{\mathcal T_y}_\rho\left(x\right)$ and $\mathfrak g^{\mathcal T_y}_\rho\left(x\right)$ are well defined and further  $-K \leq \mathfrak g^{\mathcal T_x}_\rho\left(x\right)$, since immediate stopping is allowed in the case $x=y$. 
	\end{remark}
	

\begin{defn}\label{grossgdefinition}
	We define $$\mathfrak G:\R\rightarrow [-K,\infty]; \ \rho \mapsto \sup_{x\in B} \mathfrak g^{\mathcal T_x}_\rho\left(x\right).$$ 
\end{defn}
\begin{lemma}
	$\mathfrak G$ is decreasing, on $\mathfrak{G}^{-1}\left(\left(-K,\infty\right)\right)$ even strictly decreasing, convex and continuous on $\R\setminus \{\beta\}$ where $\beta:=\inf\{\rho \in \R|\mathfrak G\left(\rho\right)=\infty\}$ (with the convention $\inf \emptyset=\infty=-\sup \emptyset$). 
\end{lemma}
\begin{proof} 
	The monotonicity is clear, $\mathfrak G$ is convex as supremum over affine functions, hence also continuous on $ \R\setminus \{\beta\}$. 
\end{proof}
We define \begin{align}
\rho^*:=\sup\{\rho\in \R|\mathfrak G\left(\rho\right)> 0\} \label{rhostardefEQ}
\end{align} and note that due to the monotonicity $$\rho^*=\inf\{\rho\in \R|\mathfrak G\left(\rho\right)< 0\}$$ and if $\rho^* \neq \beta$, $\rho^*$ is the only root of $\mathfrak G$. Now in the following we show that $v\left(y\right)=\rho^*$ for all $y\in \R$.
\begin{thm}\label{stoppcharakterisierung}
	For all $\rho\in \R$ with $\mathfrak G\left(\rho\right) \in \R$ holds $$\mathfrak G\left(\rho\right)>0 \Leftrightarrow \forall x\in \R: v\left(x\right)>\rho.$$
\end{thm}
\begin{proof} If there is $x \in B$ and a stopping time $\tau \in \mathcal T_x$, such that $\E_x\left(\gamma\left(X_\tau\right)\right)= \infty$, then Lemma \ref{endlichkeitslemma} yields the equivalence, therefore in the following we assume no such stopping time exists. 
First, let $\rho \in \R$ such that for all $x\in \R$ holds $v\left(x\right)>\rho$. Then there is an admissible strategy $S=\left(\tau_n,\xi_n\right)\in \mathcal S_B$ such that $$\liminf_{T\rightarrow \infty}\frac{1}{T} \E_x\left\{\sum_{1\leq n:\tau_n\leq T}\left(\gamma\left(X^S_{\tau_n,-}\right)-\gamma\left(\xi_n\right)-K \right)-\int^{T}_0 h\left(X^S_s\right)\ ds\right\}>\rho$$ and due to excluding strategies with an infinite gain in one period or infinite costs in one period, we also have
$$\liminf_{T\rightarrow \infty}\frac{1}{T} \E_x\left\{\sum_{1\leq n:\tau_{n-1}\leq T}\left(\gamma\left(X^S_{\tau_n,-}\right)-\gamma\left(\xi_{n-1}\right)-K -\int^{\tau_n}_{\tau_{n-1}} h\left(X^S_s\right)\ ds\right)\right\}>\rho$$
 because we sum over one more control and hence add one summand with finite expectation. We set $\tau_0:=0$ and $\xi_0:=x$. \\
 This implies that there is $\tilde T >0$ such that for all $T\geq \tilde T$  we have $$\frac{1}{T}\E_x\left\{\sum_{1\leq n:\tau_{n-1}\leq T}\left(\gamma\left(X_{\tau_n,-}\right)-\gamma\left(\xi_{n-1}\right)-K -\int^{\tau_n }_{\tau_{n-1}} h\left(X_s\right)\ ds\right)\right\}>\rho $$ and hence
 $$\frac{1}{T}\E_x\left\{\sum_{1\leq n:\tau_{n-1}\leq T}\left(\gamma\left(X_{\tau_n,-}\right)-\gamma\left(\xi_{n-1}\right)-K -\int^{\tau_n }_{\tau_{n-1}} \left(h\left(X_s\right)+\rho\right) \ ds\right)\right\}>0.$$ 
For all $n \in \N$ set $\tilde \tau^n_i =\begin{cases}
\tau_i; \ \ i\leq n\\ \infty ; \ \ i>n
\end{cases} $ and $\tilde S_n:=\left(\tilde \tau^n_i,  \xi_i\right)_{i \in \N}$. Although this is not an admissible impulse control strategy, we still used the established notations for these strategy.\\ Fix $n \in \N$. The process $Y$ given by $Y_t:=X^{\tilde S_{n-1}}_{\tau_{n-1}+t}$ is still a Markov process (started in $X^{\tilde S_{n-1}}_{\tau_{n-1}}$) both under its natural filtration $\mathcal F^Y$ and under the filtration $\mathcal{\tilde F}$ given by $\mathcal{\tilde F}_t:=\mathcal F_{\tau_{n-1} +t}$. It is well established that in optimal stopping problems for Markov processes the value of the problem does not change, when one only considers optimization over first entry times, which are in the natural filtration of the process. We define \begin{align*}
\mathcal S_x\left(\mathcal G\right):=\{\tau| \tau\text{ is } \mathcal G \text{ st. time, } Y_\tau\geq x\text{ and } \E_x\left(\int_{0}^{\tau}\left(h\left(Y_t\right)+\rho\right) \ dt\right),\ \E_x\left(\tau\right)<\infty \}
\end{align*} for each $x\in \R$, and each $\mathcal G \in \{\mathcal{\tilde{F}}, \mathcal F^Y \}$. We set 
\begin{align*}
\sigma:=\mathds{1}_{\{ \tau_{n-1}\leq T\}}\left(\tau_n-\tau_{n-1}\right) \in \mathcal S_x\left(\mathcal{\tilde F}\right)
\end{align*} and we have due to the aforementioned reason
\begin{align*}
\E_x&\left\{\mathds{1}_{\{ \tau_{n-1}\leq T\}}\left(\gamma\left(X_{\tau_n,-}\right)-\gamma\left(\xi_{n-1}\right)-K -\int^{\tau_n}_{\tau_{n-1}} h\left(X_s\right)\
 ds\right)\right\}
 \\&=\E_x\left\{\E_{x}\left[\left.\mathds{1}_{\{ \tau_{n-1}\leq T\}}\textcolor{white}{\int_{\tau_{\bar x}}^{T^T}}\right.\right.\right.
 \\&\left.\left.\left.\textcolor{white}{\int_{\tau_{\bar x}}^{T^T}}\left(\gamma\left(X_{\tau_n,-}\right)-\gamma\left(\xi_{n-1}\right)-K -\int^{\tau_n }_{\tau_{n-1} } h\left(X_s\right)\
 ds\right)\right| \mathcal F_{\tau_{n-1}}\right]\right\}
  \\&=\E_x\left\{\mathds{1}_{\{ \tau_{n-1}\leq T\}}\textcolor{white}{\int_{\tau_{\bar x}}^{T^T}}\right.
  \\&\left.\textcolor{white}{\int_{\tau_{\bar x}}^{T^T}}\E_{x}\left[\left.\left(\gamma\left(Y_{{({\tau_n,-})-\tau_{n-1}}}\right)-\gamma\left(Y_0\right)-K -\int^{{({\tau_n,-})-\tau_{n-1}}}_{0} h\left(Y_s\right)\
 ds\right)\right|\mathcal F_{\tau_{n-1}}\right]\right\}
 \\&=\E_x\left\{\mathds{1}_{\{ \tau_{n-1}\leq T\}}\E_{X_{\tau_{n-1}}}\left(\gamma\left(Y_\sigma\right)-\gamma\left(Y_0\right)-K -\int^{\sigma }_{0} h\left(Y_s\right)\
 ds\right)\right\}
 \\&\leq \E_x\left\{\mathds{1}_{\{ \tau_{n-1}\leq T\}}\textcolor{white}{\int_{\tau_{\bar x}}^{T^T}}\right.
 \\&\left.\textcolor{white}{\int_{\tau_{\bar x}}^{T^T}}\sup_{\tau \in \mathcal \mathcal S_{X_{\tau_{n-1}}}\left(\mathcal{\tilde F}\right) }\E_{X_{\tau_{n-1}}}\left(\gamma\left(Y_\tau\right)-\gamma\left(Y_0\right)-K-\int_{0}^{\tau}\left(h\left(Y_t\right)+\rho\right) \ dt \right)\right\}
 \\&= \E_x\left\{\mathds{1}_{\{ \tau_{n-1}\leq T\}}\textcolor{white}{\int_{\tau_{\bar x}}^{T^T}}\right.
 \\&\left.\textcolor{white}{\int_{\tau_{\bar x}}^{T^T}}\sup_{\tau \in \mathcal S_{X_{\tau_{n-1}}}\left(\mathcal{ F}^Y\right) }\E_{X_{\tau_{n-1}}}\left(\gamma\left(Y_\tau\right)-\gamma\left(Y_0\right)-K-\int_{0}^{\tau}\left(h\left(Y_t\right)+\rho\right) \ dt \right)\right\}
  \\&\leq \E_x\left(\mathds{1}_{\{ \tau_{n-1}\leq T\}}\mathfrak{G}\left(\rho\right) \right)
\\&= \Pro_x\left( {\tau_{n-1}\leq T} \right)\mathfrak{G}\left(\rho\right),
\end{align*}

hence for all $T\geq \tilde T$\begin{align*}0<&\frac{1}{T}\E_x\left\{\sum_{1\leq n:\tau_{n-1}\leq T}\left(\gamma\left(X_{\tau_n,-}\right)-\gamma\left(\xi_{n-1}\right)-K -\int^{\tau_n }_{\tau_{n-1}} h\left(X_s\right)+\rho\ ds\right)\right\}\\=&\frac{1}{T}\sum\limits_{n \in \N}\Pro_{x}\left(\tau_{n-1}<T\right) \mathfrak G \left(\rho\right) \end{align*}  
and we get $$0< \mathfrak G\left(\rho\right).$$
Now we show the reverse implication. \\
Let $\rho \in \R$ such that $\mathfrak G\left(\rho\right)>0$. Then there is $y\in B$ and $\tau \in \mathcal T_y$ with
\begin{align*}
\E_y\left(\gamma\left(X_\tau\right)-\gamma\left(y\right)-K-\int_0^\tau \left(h\left(X_s\right)+\rho\right) \ ds\right)>0.\tag*{(*)}\label{asd}
\end{align*}
	We set $$S^\uparrow:=\left(\tau_n,y\right)_{n \in \N}:=R\left(\tau,y\right),$$ 
and we define $$R_i:=\gamma\left(X_{\tau_i}\right)-\gamma\left(y\right)-K-\int_{\tau_{i-1}}^{\tau_i} \left(h\left(X_s\right)\right) \ ds .$$
Then $$\E_y\left(R_i\right)>\rho  \E_y\left(\tau\right)$$ and Lemma \ref{rere} yields
\begin{align*}
\liminf_{T\rightarrow \infty}&\frac{1}{T} \E_x\left\{\sum_{1\leq n:\tau_n\leq T}\left(\gamma\left(X^{S^\uparrow}_{\tau_n,-}\right)-\gamma\left(\xi_n\right)-K \right)-\int^{T}_0 h\left(X^{S^\uparrow}_s\right)\ ds\right\}\\&=
\lim_{T\rightarrow \infty} \frac{\E_y\left(\sum_{\tau_n \leq T}R_i\right)}{T}\\&= \frac{\E_y\left(R_1\right)}{\E_y\left(\tau\right)}\\& >\rho.
\end{align*}

\end{proof}
\begin{corollary}\label{vgleichrhosterncorrolar}
The value function $v$ defined in \eqref{vdefEQ} is constant and it holds	$v=\rho^*$, where $\rho^*$ is defined in \eqref{rhostardefEQ}.
\end{corollary}
\begin{corollary}\label{optimalitycorro}
	If $\mathfrak G\left(\rho^*\right)=0$ and there are $y\in B$ and $\tau\in \mathcal T_y$ such that  $$\mathfrak G\left(\rho^*\right)=\E_y\left(\gamma\left(X_\tau\right)-\int_{0}^{\tau}\left(h\left(X_t\right)+\rho^*\right) \ dt \right)-\gamma\left(y\right)-K,$$ then the strategy $R\left(\tau,y\right)$ is optimal for $v$. 
\end{corollary}
\begin{proof}
For each $\rho < \rho^*$ holds $\mathfrak{G}\left(\rho\right)>0$ and $y$ and $\tau$ fulfil $\ref{asd}$ in the previous proof of Theorem \ref{stoppcharakterisierung}. Hence the calculations therein show that  $J_y\left(R\left(\tau,y\right)\right)>\rho$ for all $\rho<\rho^*$, hence $J_y\left(R\left(\tau,y\right)\right)\geq \rho^*=v$, which means that $\left(R\left(\tau,y\right)\right)$ is an optimizer for $v$. 
\end{proof}
\section{Direct Solution of the Stopping Problem}\label{sec:OSP}

Corollary \ref{optimalitycorro} stresses the importance of finding an optimal strategy for stopping problems of the form 
$$
g^{\mathcal T_x}_\rho\left(x\right):=\sup_{\tau \in \mathcal T_x}\E_x\left(\gamma\left(X_\tau\right)-\int_{0}^{\tau}\left(h\left(X_t\right)+\rho\right) \ dt \right) 
$$where $ x\in B$, ideally optimizers of a simple form. In the following we fix $\rho \in \R$ and establish sufficient conditions for a threshold time $\tau_a$ to be an optimizer for $g^{\mathcal T_{x}}\left(x\right)$ and partition the state space $\R$ in a set where this threshold time is the optimal one and a set where immediate stopping is optimal. 
 The main tool to characterize the payoff functions, or more precisely the pairs $\gamma,h$ of payoff function and running costs, for this to hold, is a representation of $\gamma$ in terms of expected running suprema. Since in contrast to other problems, like e.g. \cite{mordeckisalminen} \cite{novikovshiryaevstopping07}, \cite{surya07} and \cite{CST13}, where different forms of maximum representations are used, here we cannot rely on the resolvent or e.g. Green kernels, since there is no discounting or killing in our problem. To circumvent this obstacle, we first fix $\bar y \in \R$ and develop the maximum representation, find a solution, etc. only on $\left(-\infty, \bar y\right]$ and later on show that our obtained optimizers in fact don't depend on the particular choice of $\bar y$ if $\bar y$ is chosen large enough.\\
 First aim of this section is to establish the needed framework to state the aforementioned relationship of $\gamma$ and $h$. 
  The second part of this section is devoted to establish the maximum representation. In the third part of the section this is used to prove optimality of a threshold time and also to characterize the threshold. \\
\begin{assumption}\label{maxdst} We assume that there is a function $f$ such that: \begin{enumerate}

\item For all $x  \leq {\bar y}$ $$
		\gamma\left(x\right)=\E_{x}\left[\int_{0}^{\tau_{\bar y}}f\left(\sup_{ r\leq t} X_r\right)dt\right]+\E_{x}\left[\gamma\left(X_{\tau_{{\bar y}}}\right)-\int_{{ 0}}^{\tau_{\bar y}} h\left(X_s\right) ds\right].$$
\item The function $f$ has a unique minimum $a \in \R$, is strictly decreasing on $\left(-\infty, a\right]$ and strictly increasing on $[a,\infty)$. 
			\end{enumerate}
\end{assumption}

We remark that $\left(\sup_{ r\leq t }X_r\right)_{t\geq 0}$ is no Markov process, but the two dimensional process $\left(X_t,\sup_{ r\leq t }X_r\right)_{t\geq 0}$ is. Hence whenever the running supremum occurs together with the measure $ \Pro_x$ or the corresponding expectation operator we tacitly mean the measure $\Pro_{\left(x,x\right)}$, the measure corresponding to the two dimensional Markov process $\left(X_t,\sup_{ r\leq t }X_r\right)_{t\geq 0}$ started in $\left(x,x\right)$. 
 \subsection{Solution of stopping problem}
We introduce the notation $$f_\rho:= f+ \rho$$ for each $\rho \in \R$ and for this section fix $\rho \in \R$ such that: 
\begin{assumption}\label{fassumption}
 The function $f_\rho$ has two roots $\underline{x} < \overline x$ and is negative on $\left(\underline x, \overline x\right)$. 
\end{assumption} 
Further, we restrict ourselves to the case of arbitrary downshifts allowed. Hereby we remark that the case of just one possible restarting point as it is assumed in e.g. forest management (see \cite{AL}) is also covered by our line of argument with minor adjustments. Our techniques even work for many more sets of $B$, nevertheless, the additional technical complexity would blur the underlying idea more than we would benefit from the greater generality. 
\begin{assumption}
From now on we assume $B=\R$.
\end{assumption}
\begin{defn}
	Call a stopping time $\tau$ upper regular if there is $z \in \R$ such that $\tau$ is under all $\Pro_x$ a.s. bounded by the first entry time of $X$ in $[z,\infty)$. For all $x\in \R$ define $\mathcal U_x:=\{\tau \in \mathcal T_x|\tau \text{ is upper regular}\}$. 
\end{defn}
\begin{lemma}\label{regularlemma}
	For all $x\in \R$ we have $$g^{\mathcal T_x}_\rho\left(x\right)=\sup_{\tau \in \mathcal U_x}\E_x\left(\gamma\left(X_\tau\right)-\int_0^\tau (h\left(X_s\right) +\rho) \ ds\right).$$
\end{lemma}
\begin{proof}
	Let $\epsilon >0$. Let $x\in \R$. Take $\tau \in \mathcal T_x$ that is $\epsilon$-optimal. Set for all $n \in \N$ \[\sigma_n:=\tau \wedge \inf\{t\geq 0\mid X_t \geq n+x\}.   \] 
	We have $\sigma_n\in \mathcal U_x$ and $\sigma_n\rightarrow \tau$ a.s. under $\Pro_x$ and since $ \int_0^\tau h\left(X_s\right)  \ ds$ works as an integrable majorant, we get with dominated convergence \[
	\E_x\left(\int_0^{\tau} h\left(X_s\right) \ ds\right)=\lim_{n\rightarrow \infty}\E_x\left(\int_0^{\sigma_n}h\left(X_s\right)  \  ds\right)
	\]
	and by monotone convergence
	\[\lim_{n\rightarrow \infty}\E_x\sigma_n=\E_x\tau.\]
	Note that due to $\tau, \sigma_n \in \mathcal T_x$ for all $n \in \N$ we have for all $ n \in \N$ that $$\gamma\left(X_\tau\right)\wedge \gamma\left(X_{\sigma_n}\right)\geq \gamma(x).$$ Thus Fatou's Lemma yields \begin{align*}
	\E_x \left(\gamma\left(X_\tau\right)\right)&=\E_x\left(\lim_{n\rightarrow \infty}\gamma\left(X_\tau\right)\wedge \gamma\left(X_{\sigma_n}\right)\right)\\&\leq \liminf_{n\rightarrow \infty}\E_x\left(\gamma\left(X_\tau\right)\wedge \gamma\left(X_{\sigma_n}\right)\right)\\&\leq\liminf_{n\rightarrow \infty}\E_x\left( \gamma\left(X_{\sigma_n}\right)\right).
	\end{align*}
	Altogether we get
	\[
	\E_x\left(\gamma\left(X_\tau\right)-\int_0^{\tau}(h\left(X_s\right) +\rho) \ ds\right)\leq\liminf_{n\rightarrow \infty}\E_x\left(\gamma\left(X_{\sigma_n}\right)-\int_0^{\sigma_n} (h\left(X_s\right) +\rho)\ ds\right).
	\]
\end{proof}

\begin{lemma}\label{39}
	Let $x,\bar y \in \R$ with $\underline x \leq x\leq \bar y$. For $a_{\bar y}:=\bar y \wedge \bar x$ holds
	$$\sup_{  \tau \leq  \tau_{\bar y}}\E_{x}\left[\int_{0}^{\tau}(-f_\rho)\left(\sup_{ r\leq t} X_r\right)dt\right]=\E_{x}\left[\int_{0}^{\tau_{a_{\bar y}}}(-f_\rho)\left(\sup_{r \leq t} X_r\right)dt\right].$$ 
\end{lemma}
\begin{proof}
This is a direct consequence of the properties of $f$ posed upon it in Assumption \ref{maxdst}, Assumption \ref{fassumption} and the monotonicity of $\sup_{r \leq t}X_r$. 
\end{proof}

%
In the following we use the notation $
f^-(x):=-\min\{f(x),0\}
$ for all $x\in \R$. 
\begin{lemma}\label{bestex}
	There is $x^*\in [\underline x, \bar x] $ such that for all $x \in \R$ \[
	\E_{x}\left[\int^{\tau_{\bar x}}_0  f_\rho^-\left(\sup_{ r\leq t }X_r\right)dt \right]\leq \E_{x^*}\left[\int^{\tau_{\bar x}}_0  f_\rho^-\left(\sup_{ r\leq t }X_r\right)dt \right]. \]
\end{lemma}
\begin{proof}
	Since $[\underline x, \bar x]$ is compact and the function $\R\rightarrow \R; \ x\mapsto \E_{x}\left[\int^{\tau_{\bar x}}_0  f_\rho^-\left(\sup_{ r\leq t }X_r\right)dt \right]$ is continuous due to Lemma \ref{XIcontLEMlevy} (which is applicable because Lemma \ref{supdarstellung} allows us to replace the running maximum with the ladder height process), $$x^*\in \mathrm{argmax}_{x\in [\underline x, \bar x] }	\E_{x}\left[\int^{\tau_{\bar x}}_0  f_\rho^-\left(\sup_{ r\leq t }X_r\right)dt \right]$$ exists. Now for each $x \geq \bar x$ we have 
	\begin{align*}
	\E_{x}\left[\int^{\tau_{\bar x}}_0  f_\rho^-\left(\sup_{ r\leq t }X_r\right)dt\right] = 0
	=\E_{\bar x}\left[\int^{\tau_{\bar x}}_0  f_\rho^-\left(\sup_{ r\leq t }X_r\right)dt\right] .
	\end{align*}
	And for each $x\leq \underline x$ we get 
	\begin{align*}
	\E_{x}&\left[\int^{\tau_{\bar x}}_0  f_\rho^-\left(\sup_{ r\leq t }X_r\right)dt\right]
	\\&= \E_{x}\left[\int^{\tau_{\bar x}}_{\tau_{\underline x}}  f_\rho^-\left(\sup_{ r\leq t }X_r\right)dt\right] 
	\\&= \E_{x}\left[\E_{x}\left[\int^{\tau_{\bar x}}_{0}  f_\rho^-\left(\sup_{ r\leq t }X_r\right)dt\big|\mathcal F_{\tau_{\underline x}}\right] \right]
	\\&= \E_{x}\left[\E_{X_{\tau_{\underline x}}}\left[\int^{\tau_{\bar x}}_{0}  f_\rho^-\left(\sup_{ r\leq t }X_r\right)dt\right] \right]
	\\&\leq  \E_{x}\left[\E_{x^*}\left[\int^{\tau_{\bar x}}_{0}  f_\rho^-\left(\sup_{ r\leq t }X_r\right)dt\right] \right]
	\\&=\E_{x^*}\left[\int^{\tau_{\bar x}}_{0}  f_\rho^-\left(\sup_{ r\leq t }X_r\right)dt\right].
	\end{align*}
\end{proof}

 \begin{thm}\label{stoppinghauptsatz}
	For all $x \in \R$ with $x \geq \underline x$ holds $$ g^{\mathcal T_{x}}_\rho\left(x\right)=\E_x\left(\gamma\left(X_{\tau_{\bar x}}\right)-\int_{0}^{\tau_{\bar x}}\left(h\left(X_s\right)+\rho\right)\ ds\right) $$
	For all $x\in \R$ with $x<\underline x$, there is $x^*\in [\underline x,\overline x]$ such that $$g^{\mathcal T_{x}}_\rho\left(x\right)\leq g^{\mathcal T_{x^*}}_\rho\left(x^*\right).$$ 
\end{thm}
\begin{proof}
We define for all $x\in \R$ $$\tilde g\left(x\right):=\E_x\left(\gamma\left(X_{\tau_{\bar x}}\right)-\int_{0}^{\tau_{\bar x}}(h\left(X_s\right)-\rho) \ ds\right). $$ One immediately sees that $ g^{\mathcal T_x}_\rho\geq \tilde g(x) \geq \gamma(x)$ for all $x\in \R$.
Let $x \in \R$. Lemma \ref{regularlemma} tells us that it suffices to show $$\tilde g\left(x\right)\geq \sup_{  \tau \in \mathcal U_x} \E_x\left(\gamma\left(X_\tau\right)-\int_0^{\tau}\left( h\left(X_s\right)+\rho \right) \ ds\right)$$ in order to prove $g_\rho^{\mathcal T_x}=\tilde g(x)$.
\\
  Let $\tau\in \mathcal U_x$ be an upper regular stopping time and fix ${\bar y}>x,\overline x$ such that $\tau\leq \tau_{\bar y}$ $\Pro_x$ a.s. Then we have
 
\begin{align*}
 \E_x &\left[\gamma\left(X_\tau\right)-\int_0^{\tau}\left( h\left(X_s\right)+\rho \right) \ d s\right]
 \\\stackrel{\ref{maxdst}}{=}&\E_x\left\{\E_{X_\tau}\left[\int_{0}^{\tau_{\bar y}}f_\rho\left(\sup_{ r\leq t}X_r\right)dt\right]\right.
 \\&\left.+\E_{X_\tau}\left[\gamma\left(X_{\tau_{{\bar y}}}\right)-\int_{{ 0}}^{\tau_{\bar y}} \left( h\left(X_s\right)+\rho \right) \ ds\right] -\left[\int_0^{\tau}\left( h\left(X_s\right)+\rho \right)\ d s\right]\right\}
 \\=&\E_x\left\{\E_{X_\tau}\left[\int_{0}^{\tau_{\bar y}}f_\rho\left(\sup_{ r\leq t}X_r\right)dt\right]\right\}+\E_{x}\left[\gamma\left(X_{\tau_{{\bar y}}}\right)-\int_{0}^{\tau_{\bar y}} \left( h\left(X_s\right)+\rho \right)\ ds \right]
  \\=&\E_x\left\{\E_{x}\left[\int_{\tau}^{\tau_{\bar y}}f_\rho\left(\sup_{ r\leq t}X_r\right)dt|\mathcal F_\tau \right]\right\}+\E_{x}\left[\gamma\left(X_{\tau_{{\bar y}}}\right)-\int_{0}^{\tau_{\bar y}} \left( h\left(X_s\right)+\rho \right)\ ds \right]
  \\=&\E_x\left\{\mathds{1}_{\{\tau \leq \tau_{\underline x}\}}\E_{x}\left[\int_{\tau}^{\tau_{\bar y}}f_\rho\left(\sup_{ r\leq t}X_r\right)dt|\mathcal F_\tau \right]\right\}
  \\&+\E_x\left\{\mathds{1}_{\{\tau > \tau_{\underline x}\}}\E_{x}\left[\int_{\tau}^{\tau_{\bar y}}f_\rho\left(\sup_{ r\leq t}X_r\right)dt|\mathcal F_\tau \right]\right\}+\E_{x}\left[\gamma\left(X_{\tau_{{\bar y}}}\right)-\int_{0}^{\tau_{\bar y}} \left( h\left(X_s\right)+\rho \right)\ ds \right]
   \\=&\E_x\left\{\E_{x}\left[\mathds{1}_{\{\tau \leq \tau_{\underline x}\}}\int_{\tau}^{\tau_{\bar y}}f_\rho\left(\sup_{ r\leq t}X_r\right)dt|\mathcal F_\tau \right]\right\}
  \\&+\E_x\left\{\E_{x}\left[\mathds{1}_{\{\tau > \tau_{\underline x}\}}\int_{\tau}^{\tau_{\bar y}}f_\rho\left(\sup_{ r\leq t}X_r\right)dt \right]|\mathcal F_\tau \right\}
  \\&+\E_{x}\left[\gamma\left(X_{\tau_{{\bar y}}}\right)-\int_{0}^{\tau_{\bar y}} \left( h\left(X_s\right)+\rho \right) \ ds \right]
   \\\stackrel{\ref{maxdst}}{\leq}&\E_{x}\left[\mathds{1}_{\{\tau \leq \tau_{\underline x}\}}\int_{0}^{\tau_{\bar y}}f_\rho\left(\sup_{ r\leq t}X_r\right)dt\right]
  \\&+\E_{x}\left[\mathds{1}_{\{\tau > \tau_{\underline x}\}}\int_{\tau_{\bar x}}^{\tau_{\bar y}}f_\rho\left(\sup_{ r\leq t}X_r\right)dt\right]\\&+\E_{x}\left[\gamma\left(X_{\tau_{{\bar y}}}\right)-\int_{0}^{\tau_{\bar y}} \left( h\left(X_s\right)+\rho \right)\ ds \right]
   \\\stackrel{\ref{maxdst}}{=}&\E_{x}\left[\mathds{1}_{\{\tau \leq \tau_{\underline x}\}}\int_{0}^{\tau_{\bar y}}f_\rho\left(\sup_{ r\leq t}X_r\right)dt \right]
  \\&+\E_{x}\left[\mathds{1}_{\{\tau > \tau_{\underline x}\}}\int_{\tau_{\bar x}}^{\tau_{\bar y}}f_\rho\left(\sup_{ r\leq t}X_r\right)dt \right]+\gamma\left(x\right)-\E_x\left[\int_{{ 0}}^{\tau_{\bar y}}f_\rho\left(\sup_{r\leq t}X_{r}\right) \ dt\right]
  \\=&\gamma\left(x\right)+\E_{x}\left[\mathds{1}_{\{\tau > \tau_{\underline x}\}}\int^{\tau_{\bar x}}_0 - f_\rho\left(\sup_{ r\leq t}X_r\right)dt \right]
  \\\leq& \gamma\left(x\right)+\E_{x}\left[\mathds{1}_{\{\tau > \tau_{\underline x}\}}\int^{\tau_{\bar x}}_0 f_\rho^-\left(\sup_{ r\leq t }X_r\right)dt \right]
  \\ =&: \star
\end{align*}
Now we distinguish the cases $x \leq \underline x$ and $x >\underline x$. 

If $x > \underline x$ applying Assumption \ref{maxdst} yet another time yields
\begin{align*}
\star&\leq \gamma\left(x\right)+\E_{x}\left[\int^{\tau_{\bar x}}_0 f_\rho^-\left(\sup_{ r\leq t }X_r\right)dt \right]
\\&=\E_{x}\left[\gamma\left(X_{\tau_{ \bar x}}\right)-\int_0^{\tau_{\bar x}} \left( h\left(X_s\right)+\rho \right)  \ ds\right].
\end{align*}

If $x\leq \underline x$ due to the monotonicity of $\gamma$ we have using Lemma \ref{bestex} and the definition of $x^*\in [\underline x, \bar x]$ therein 
\begin{align*}
\star&\leq \gamma\left(x^*\right)+\E_{x^*}\left[\int^{\tau_{\bar x}}_0  f_\rho^-\left(\sup_{ r\leq t }X_r\right)dt \right]
\\&=\E_{x^*}\left[\gamma\left(X_{\tau_{ \bar x}}\right)-\int_0^{\bar x} h\left(X_s\right)+\rho \ ds\right]\\&= \tilde g(x^*)
\\&\leq g^{\mathcal T_{x^*}}_\rho (x^*).
\end{align*}
\end{proof}

\section{The Optimal Restarting Point}\label{sec:restarting}
So far we characterized the (random) optimal times to exercise controls by optimal stopping times of an associated stopping problem. Assuming a supremum representation of the payoff function $\gamma$ that involves a sufficiently favourably shaped function $f$, the characterization boils down to exercise a control, whenever the process exceeds the  rightmost root of the function $f_{\rho^*}$ defined in and after Assumption \ref{maxdst}. The optimal restarting point so far only is characterized as the optimizer of $$\sup_{y\in \R}\mathfrak g_{\rho^*}^{\mathcal T_y}\left(y\right).$$ Now we show that if the ladder height process $H$ is a special subordinator for fixed $\rho\in \R$  the lower root of $f_\rho$, $\underline x$ is indeed the maximizer of $$\mathfrak G\left(\rho\right)=\sup_{y\in \R}\mathfrak g_\rho^{\mathcal T}.$$ If the Assumption \ref{fassumption} also particularly holds for $\rho^*$, it immediately follows that the $\left(s,S\right)$ strategy with $s=\underline x$ and $S=\overline x$ is optimal for the control problem. 
Having worked extensively with the maximum of $X$ so far, now the use of the ladder height process turns out to be more handy. Hence with the first lemma, we establish a connection between expected integral over the first and over the latter. \\
Again we fix $\rho \in \R$ throughout the section and assume $f_\rho$ has precisely two roots $\underline x < \overline x$.
\begin{lemma}\label{supdarstellung}  For all $x < y$ and all measurable functions $\varphi$ such that the following expressions exist holds
	$$\E_x\left[\int_{0}^{\hat \tau_{y}}\varphi\left(H_s\right) \ ds\right]{=}\E_x\left[\int_{0}^{\tau_{ y}}\varphi\left(\sup_{r\leq t} X_r\right)\ ds \right].$$
\end{lemma}
\begin{proof}
	
	This can be proven by algebraic induction: Wald's identity shows $$\E_x\left(\hta_y\right)=\E\left(L^{-1}_1\right)\E_x\left(\tau_y\right)=\E_x\left(\tau_y\right),$$ hence the claim holds for indicator functions of the form $\mathds 1_{[x,y]}$ and with the Markov property this extends to indicator functions of general intervals. This carries over to simple positive functions due to linearity and with Fatou's lemma to general positive functions. Decomposition in a positive and a negative part yields the claim for general measurable functions.
\end{proof}
The following definitions and results can be found in \cite{schillingsongvondracek}, but also Section 5.6 in \cite{kypri} provides an overview over Bernstein function, that is rather L\'evy process centred. 
\begin{defn}\label{specialsub}
	Let $S$ be a subordinator with Laplace exponent $$\phi:\left(0,\infty\right)\rightarrow \R; \ \lambda \mapsto a+b\lambda +\int_0^\infty\left(1-e^{-\lambda t}\right)\ \mu\left(dt\right) $$ for $a,b \geq 0$ and $\sigma$-finite measure $\mu$ on $\left(0,\infty\right)$ with $\int_0^\infty \left(t \wedge 1\right) \mu\left(d t\right)\leq \infty$. Then $S$ is called a special subordinator if $\phi$ is a special Bernstein function, i.e. $\tilde \phi:= \frac{\mathrm{id}}{\phi}$ is also the Laplace exponent of a subordinator. 
\end{defn}
\begin{lemma}\label{specialcharacterization}
Let $S$ be a subordinator with potential measure $U$. Then $S$ is special if and only if $U|_{\left(0,\infty\right)}$ has a decreasing density $u$ with $\int_0^1u\left(t\right)dt < \infty$. 
\end{lemma} 
\begin{proof} 
This is Theorem 5.19 in \cite{kypri}. 
\end{proof}
\begin{remark}\label{specialexamples}
Many common examples of subordinators are special, including: 
\begin{enumerate}
	\item All stable subordinators. 
	\item Each subordinator whose jump measure has a log convex density. 
	\item\label{specialexamples:c} Each subordinator whose jump measure has a completely monotone density.
	\item Each subordinator whose L\'evy measure $\nu$ has the property that $\R \rightarrow \R; \ x\mapsto \log \nu(x,\infty)$ is a convex function. 
\end{enumerate}
\end{remark}

\begin{remark}
	Since the most favourable case for us here is that the ladder height process $H$ is a special subordinator, the question arises how one can make sure that $H$ falls in the class of special subordinators by looking at characteristics of $X$. Theorem 7.8 in \cite{kypri} yields that for each $y>0$ \begin{align*}
	\Pi_H(y,\infty)=\int_{[0,\infty)}\Pi(z+y,\infty) U^\downarrow(dz)
	\end{align*} where $\Pi$ is the L\'evy measure of $X$, $\Pi_H$ the one of $H$ and $$U^\downarrow(dz)=\E\left(\int_0^\infty \mathds{1}_{\{H^\downarrow_t\in dz\}}\ dt\right)$$ is the potential measure of the descending ladder height process. \\ Now this formula may help to verify one of the necessary conditions for $H$ to be a special subordinator stated in Remark \ref{specialexamples} by using our knowledge of  $\Pi$. Especially the condition \ref{specialexamples}, \ref{specialexamples:c} turns out to be a handy one, since if $\Pi$ has a completely monotone density, so has $\Pi_H$. And the former applies to many L\'evy processes of interest. 
\end{remark}
Now we proceed to show that provided $H$ is a special subordinator, $\underline x$ is the optimal restarting point.


\begin{thm}\label{xuntenthm} Let $x^*\in \text{argmax}_{y\in \R}\ \mathfrak g_\rho^{\mathcal T_y}(y)$. Then 
$$x^* \geq \underline x.$$ 	If $X$ is not a compound Poisson process and $H$ is a special subordinator, then furthermore $$\underline{x} \in \text{argmax}_{y\in \R}\mathfrak g_\rho^{\mathcal T}(y). $$ 
\end{thm}
\begin{proof}
First, we show $x^* \geq \underline x$.\\
Let $x\in \R$ with $x < \underline x$. Then $$\E_x\left(\int_0^{\hat \tau_{\underline x}} f_\rho\left(H_s\right)\ ds\right)<0$$ and hence we obtain by use of Assumption \ref{maxdst} in combination with Lemma \ref{supdarstellung} \begin{align*}
\mathfrak g_\rho^{\mathcal T_x}\left(x\right)+K&=\E_x\left(\gamma\left(X_{\tau_{\overline x}}\right)-\int_{{ 0}}^{\tau_{\overline x}}\left(h\left(X_s\right) +\rho \right) \ ds\ \right)-\gamma\left(x\right)
\\ & \stackrel{\ref{maxdst}}{=}\E_x\left(\int_0^{ \tau_{\overline x}} (-f_\rho)\left(\sup_{ r\leq s }X_s\right)\ ds\right)
\\ & \stackrel{\ref{supdarstellung}}{=}\E_x\left(\int_0^{\hat \tau_{\overline x}} (-f_\rho)\left(H_s\right)\ ds\right)
\\ & =\E_x\left(\int_0^{\hat \tau_{\underline x}} (-f_\rho)\left(H_s\right)\ ds+\int_{\hat \tau_{\underline x}}^{\hat \tau_{\overline x}} (-f_\rho)\left(H_s\right)\ ds\right)
\\&=\E_x\left(\int_0^{\hat \tau_{\underline x}} (-f_\rho)\left(H_s\right)\ ds\right)+\E_x\left\{\E_{H_{\hat \tau_{\underline x}}}\left(\int_{\hat \tau_{\underline x}}^{\hat \tau_{\overline x}} (-f_\rho)\left(H_s\right)\ ds\right)\right\}\\&<\E_x\left\{\E_{H_{\hat \tau_{\underline x}}}\left(\int_{\hat \tau_{\underline x}}^{\hat \tau_{\overline x}} (-f_\rho)\left(H_s\right)\ ds\right)\right\}\\&\leq\E_a\left\{\int_{\hat \tau_{\underline x}}^{\hat \tau_{\overline x}} (-f_\rho)\left(H_s\right)\ ds\right\}
\end{align*} for some $a \in [\underline x, \bar x]$.\\

Now we show $\underline x\in \text{argmax}_{y\in \R}\ \mathfrak g_\rho^{\mathcal T_y}(y) $ under the assumption that $H$ is a special subordinator and not a compound Poisson process.\\ 
Let $U$ be the potential measure of $H$. 
Since $H$ is a special subordinator, according to Lemma \ref{specialcharacterization} $U|_{\left(0,\infty\right)}$ has a decreasing density $u$, since $X$ is not a compound Poisson process $U$ has furthermore no point mass at 0. \\

Let $x \in [\underline x, \bar x]$. Then \begin{align*}
\mathfrak g_\rho^{\mathcal T_x}\left(x\right)+K&=\E_x\left(\gamma\left(X_{\tau_{\overline x}}\right)-\int_{{ 0}}^{\tau_{\overline x}}\left(h\left(X_s\right) +\rho \right) \ ds\ \right)-\gamma\left(x\right) 
\\ & \stackrel{\ref{maxdst}}{=}\E_x\left(\int_0^{ \tau_{\overline x}} (-f_\rho)\left(\sup_{ r\leq s }X_s\right)\ ds\right)
\\ & \stackrel{\ref{supdarstellung}}{=}\E_x\left(\int_0^{\hat \tau_{\overline x}} (-f_\rho)\left(H_s\right)\ ds\right)
\\&=\int_x^{\bar x}(-f_\rho)\left(y\right) \ U\left(dy-x\right)
\\&=\int_x^{\bar x}(-f_\rho)\left(y\right) u\left(y-x\right) \ dy
\\& \leq \int_x^{\bar x} (-f_\rho)\left(y\right)u\left(y-\bar x\right)\ dy
\\&\leq \int_{\underline x}^{\bar x}(-f_\rho)\left(y\right) u\left(y-\bar x\right)\ dy
\\&= \E_{\underline x}\left[\gamma\left(X_{\tau_{\bar x}}\right)-\int_{0}^{\tau_{\bar x}}\left(h\left(X_s\right)+\rho\right) \ ds \right].
\end{align*}
These calculations yield $$\underline{x} \in  \text{argmax}_{y\in \R}\mathfrak g_\rho^{\mathcal T}(y).$$
\end{proof}

\section{Discussion of the assumptions}\label{sec:discussion_ass}
\subsection{On Assumption \ref{maxdst}}
In Assumption \ref{maxdst} we assume existence of a function $f$ such that for all $x, y \in \R$ with $x~\leq~{ y}$ $$
\gamma\left(x\right)=\E_{x}\left[\int_{0}^{\tau_{ y}}f\left(\sup_{ r\leq t} X_r\right)dt\right]+\E_{x}\left[\gamma\left(X_{\tau_{{ y}}}\right)-\int_{{ 0}}^{\tau_{ y}} h\left(X_s\right) ds\right].$$
If  $-f$ is unimodal, the stopping problem we solve in Section \ref{sec:OSP} has a threshold time as an optimizer that in the next step generates an optimal $\left(s,S\right)$ strategy for our initial impulse control problem.  In the most beneficial cases the boundaries $s$ and $S$ are given by the two solutions to $f\left(x\right)= \rho^*$ for the $\rho^*$ from Section \ref{sec:connection}. 
 The applicability of these results is inseparably intertwined not only with the existence of such a function $f$, but also relies on the explicit obtainability. To tackle these two questions will be the scope of this section.\\
  First, we will give sufficient conditions for such an $f$ in the maximum representation to exist and thereafter take some steps to the (semi-)explicit obtainability in interesting cases.
We remind  that $H^\downarrow$ denotes the descending ladder height process of $X$. \\

\begin{lemma}\label{dachfunktion}
For each positive function $g$ define for all $y \in \R$
\begin{align*}
\hat g\left(y\right) :=\E_y\left(\int\limits_0^\infty g\left( H^{\downarrow}_t\right) \ dt \right).
\end{align*}
Then for all $x\leq y$
\begin{align*}
\E_x \left(\int_0^{\tau_y} g\left(X_t\right) \ dt \right)=\E_x\left( \int_0^{\hat \tau_y}\hat g \left(H_t\right) \ dt\right).
\end{align*}
\end{lemma}
\begin{proof}
This  result is a reformulation of Exercise 7.10 in \cite{kypri} and originates in \cite{silverstein}. 
\end{proof}
\begin{remark} 
The process $H^\downarrow$ acts in law like a killed subordinator, see Theorem 6.9 in \cite{kypri}.
\end{remark}
We remind that $A_H$ is defined as the extended generator of $H$. 
\begin{assumption}
Assume $\gamma$ is in the range of the extended generator $A_H$ and Dynkin's  formula is applicable to each $\hat \tau_y$, i.e., for all $x,y\in \R$ with $x\leq y$ holds 
\[\E_x\gamma\left(H_{\hat \tau_y}\right)=\E_x\int_0^{\hat \tau_y}\left(A_H\gamma\right)\left(H_s\right)ds+\gamma\left(x\right).\]
\end{assumption}
We refer to \cite{OeksendalSulem2005}, Theorem 2.14, for a natural sufficient condition for underlying $C^2$-functions.

\begin{defn}\label{fdef}
	Define $$f:=-\left(A_H\gamma +\hat h\right).$$ \end{defn}

\begin{lemma}\label{777}
	For all $x,y\in \R$ with $x\leq y$ holds  $$
	\E_x\left[\gamma\left(X_{\tau_y}\right)-\int_0^{\tau_y} h\left(X_s\right) \ ds\right]=\E_x\left[\int_0^{\hat \tau_y}(-f)\left(H_s\right)ds\right]+\gamma\left(x\right).
	$$
\end{lemma}
\begin{proof}
	For all $\in \R$ with $x\leq y$ we have using the assumption and Lemma \ref{dachfunktion}	\begin{align*}
	\E_x&\left[\gamma\left(X_{\tau_y}\right)-\int_0^{\tau_y}\left(h\left(X_s\right)+\rho\right) \ ds\right]\\&=\E_x\left[\gamma\left(H_{\hat \tau_y}\right)-\int_0^{\hat \tau_y}\left(\hat h\left(H_s\right)+\rho\right) \ ds\right]
	\\&=\E_x\left[\int_0^{\hat \tau_y}\left(A_H\gamma+\hat h\right)\left(H_s\right)ds\right]+\gamma\left(x\right)
	\\&=\E_x\left[\int_0^{\hat \tau_y}(-f)\left(H_s\right)ds\right]+\gamma\left(x\right).
	\end{align*}
\end{proof}

As an easy corollary we get
\begin{proposition}\label{maxdarstellungshauptlemma} For all $ x,y \in \R$ with $x  \leq { y}$ holds $$
	\gamma\left(x\right)=\E_{x}\left[\int_{0}^{\tau_{  y}}f\left(\sup_{ r\leq t} X_r\right)dt\right]+\E_{x}\left[\gamma\left(X_{\tau_{{  y}}}\right)-\int_{{ 0}}^{\tau_{  y}} h\left(X_s\right) ds\right].$$
\end{proposition}
\begin{proof}
Lemma \ref{777} with combined with Lemma \ref{supdarstellung} and Lemma \ref{dachfunktion} yields
\begin{align*}
	\gamma\left(x\right)
	{=}&\E_{x}\left[\int_{0}^{\hat \tau_{  y}}f\left(H_s\right)ds\right]  +  \E_{x}\left[\gamma\left(X_{\tau_{{  y}}}\right)-\int_{{ 0}}^{\hat\tau_{  y}} \hat h\left(H_s\right) ds\right]
	\\ {=}&\E_{x}\left[\int_{0}^{\tau_{  y}}f\left(\sup_{ r\leq t} X_r\right)dt\right]+\E_{x}\left[\gamma\left(X_{\tau_{{  y}}}\right)-\int_{{ 0}}^{\tau_{  y}} h\left(X_s\right) ds\right].
\end{align*}
\end{proof}
\subsection{On Assumption \ref{fassumption}}
%
For our approach to work, it is essential that Assumption \ref{fassumption} holds for $\rho^*$, meaning, $f_{\rho^*}$ really has two distinct roots $\underline x_{\rho^*}, \overline x_{\rho^*}$ (see Step 2. of the solution technique of the following\ section). For this we basically need a unimodal form of the function $f$ in the integral type maximum representation discussed before. While one could generalize our findings by allowing $f$ to violate this assumption on a set where the values of $f$ are so large that they do not influence the line of argument for $\rho^*$, in a broader sense this shape of $f$ roughly reflects the idea of what is called monotone problems in optimal stopping (see the discussion for a related problem in \cite{ChristensenSohr2}). Therefore one cannot expect existence of an optimal threshold strategy when $f$ has an entirely different structure.

Given that $f$ has a unimodal form, the other restrictions in Assumption \ref{fassumption} are either not essential or not restrictive at all. Her we will go through them step by step and so outline possible generalizations of the results in this paper without going into detail. 
\begin{enumerate}
	\item  One reason, where that assumption could fail, is that $f$ is not continuous. Then instead of the roots one could take the points where $f$ jumps from positive to negative instead. While this would make no difference for the lower value $\underline x_{\rho^*}$ at all, the function $x\mapsto \E_x\left(\int_0^{\overline x_{\rho^*}}f_{\rho^*}(-\sup_{s\leq t}X_s)ds\right)$ is not necessarily continuous any more, which would not essentially change the results, but considerably complicate our proofs.	Furthermore, if $\gamma$ and $h$ are smooth enough, $f$ will be continuous. 
	\item It cannot happen that $-\rho^*$ lies below the graph of $f$, since the same arguments as in Theorem \ref{stoppinghauptsatz} would yield that for some $\epsilon >0$ we would have $\mathfrak G(\rho^*-\epsilon)<0$ and that is a contradiction to the definition of $\rho^*$. 
	\item If $-\rho^*=\min_{x\in \R}f(x)$ then we usually only get $\epsilon$-optimal strategies in the class of impulse control problems because the optimal strategies are of singular-type. Just in the case $\Pro_a(\tau_a\neq \mathring \tau_a)>0$ or for $a:= \text{argmin}_{x\in \R} f(x)$, always starting in $a$ and shifting the process back whenever it is strictly larger that $a$, would be an optimal impulse strategy. 
	\item The case $-\rho^* > \max\{\inf_{x<a}f(x),\inf_{x>a}f(x)\}$ again would lead to a contradiction to the definition of $\rho^*$.
	\item If $-\rho^* = \max\{\inf_{x<a}f(x),\inf_{x>a}f(x)\}$ we might not have two roots of $f_{\rho^*}$ but instead of possible missing roots we could use $-\infty$ or $\infty$, resp., and would only get $\epsilon$-optimal strategies. 
\end{enumerate}

\section{Proof of the Validity of the Solution Technique}\label{sec:proof_sol_tech}
The scope of this section is to briefly connect the dots and use our findings in order to show that the step by step solution technique presented in Section \ref{sec:main} indeed is valid. 
 \begin{enumerate}
	\item  Proposition \ref{maxdarstellungshauptlemma} shows that for\begin{align*}
	f=- \left(A_H \gamma + \hat h \right) 
	\end{align*}
	we have
	$$
	\gamma\left(x\right)=\E_{x}\left[\int_{0}^{\tau_{\bar y}}f\left(\sup_{ r\leq t} X_r\right)dt\right]+\E_{x}\left[\gamma\left(X_{\tau_{{\bar y}}}\right)-\int_{{ 0}}^{\tau_{\bar y}} h\left(X_s\right) ds\right] $$ for all $x, \bar y \in \R$ with $x \leq \bar y$, hence the desired maximum representation of $\gamma$. 
	\item The second step of the solution technique is to find $\rho^* \in \R$ such that $f\left(x\right)=\rho^* $  has exactly two solutions $\underline x_{\rho^*} < \overline x_{\rho^*}$  and \begin{align*}
	0&=\sup_{x\in [\underline x_{\rho^*}, \overline x_{\rho^*}]}\E_x\left(\gamma\left(X_{\tau_{\overline x_{\rho^*}}}\right)-\int_{{ 0}}^{\tau_{\overline x_{\rho^*}}}\left(h\left(X_s\right)+ \rho^* \right)\ ds\right)
	\\&= \sup_{x\in [\underline x_{\rho^*}, \overline x_{\rho^*}]}\E_x\left(\int_0^{\hat \tau_{\overline x}}f\left(\sup_{ r\leq t }X_r\right)- \rho^*\ dt \right).
	\end{align*}
	Assume we have found such elements. Then Theorem \ref{stoppinghauptsatz} yields that the threshold time $\tau_{\overline x_{\rho^*}}$ is an optimizer for the stopping problem with value function $g_\rho^{\mathcal T_\cdot}$. Further, the first part of Theorem \ref{xuntenthm} ensures that there is an \begin{align*}
	s&\in \arg\max_{x\in [\underline x_{\rho^*}, \overline x_{\rho^*}]}\E_x\left(\gamma\left(X_{\tau_{\overline x_{\rho^*}}}\right)-\int_{{ 0}}^{\tau_{\overline x_{\rho^*}}}\left(h\left(X_t\right)+ \rho^*\right)  \ dt\right). 
	\end{align*} For this $s$ by definition we have \begin{align*}
	\mathfrak G(\rho^*)=\mathfrak g_{\rho^*}^{\mathcal T_s}(s)
	\end{align*} where $\mathfrak{G}$ is defined in Definition \ref{grossgdefinition} and $\mathfrak{g}$ right before. 
	Now Corollary \ref{vgleichrhosterncorrolar} yields \begin{align*}
	v=\rho^*
	\end{align*} and Corollary \ref{optimalitycorro} shows that the strategy $R(\tau_{\overline x_{\rho^*}},s)$ as defined in Section \ref{threshholdtimeandstrategydefinition} is optimal. 
	
	\item The second part in Theorem \ref{xuntenthm} shows that in case that $H$ is a special subordinator and not a compound Poisson process, $\underline x$ is a valid choice for $s$. 
	Further, the sixth section gives conditions in term of properties of $X$ under that $H$ is a special subordinator. 
\end{enumerate}

\section{Applications and Examples}\label{sec:appl}

After we have proven the validity of our solution technique and already have seen some first applications in Section \ref{sec:main}, where we got existence results and good starting points for numerical analysis we now demonstrate on some prominent examples how our solution technique yields (semi-)explicit characterizations of the control boundaries. \\ 
 We focus on two classes of important applications, inventory control and optimal harvesting and use our solution technique to not only show existence of optimal $(s,S)$ strategies, but also characterize boundaries and value.

 \subsection{Inventory control for spectrally one sided L\'evy processes}\label{subsec:inventory}

The first example we treat is inventory control. We want to remark, that since in inventory control one usually seeks to minimize the costs of ordering supplies and maintaining a stock depending on a draining inventory level modelled by a process $X$, we have  to turn the  usual setting of inventory control (see e.g.\cite{yamazaki} or \cite{helmes2017continuous}) 'upside down' to translate it to the maximization problem we treat herein. 
Although the majority of inventory control problems uses discounting cost and payoff functionals, we want to emphasize that in inventory control the long term average reward is of no less interest compared to the discounted one. For instance in \cite{helmes2017continuous} and \cite{helmes2017weak} the authors show optimality for $\left(s,S\right)$ strategies in the long term average problem under roughly similar conditions as ours here, provided the underlying process $X$ is a diffusion, after they obtained comparable results in the case with discounted payoff in \cite{helmes2015discounted}. However, the lack of existence of $0$-resolvent and the often constant value function makes it impossible to directly adapt techniques from the discounted case in the long term average one. Usual results in inventory control prove existence of optimal $(s,S)$ strategies and sometimes even characterize the optimal boundaries as maximizers of some functionals given by parameters of the process. In the long term average setting such results so far are only present for diffusions, see \cite{helmes2017continuous}. Therein the optimal values are given as optimizers of a functional that consists of integrals over speed measure and scale function of the underlying diffusion, see Proposition 3.5 in \cite{helmes2017continuous}. 

Although there are yet very few comparable results for L\'evy processes, over the course of the last decade the theory of scale functions for spectrally one sided L\'evy processes gave rise to new advances in control theory of these processes. In inventory control in \cite{yamazaki} Yamazaki applied these techniques to show optimality of an $\left(s,S\right)$ strategy when the reward functional is a discounted one and characterized the boundaries as optimizers of a certain functional by use of the scale functions  under roughly the following conditions: 
\begin{enumerate}
	\item The process X drifts upwards, is spectrally positive and $\E\left(\exp\left(\beta X_1\right)\right)<\infty$ for some $\beta>0$. 
	\item The payoff function $\gamma$ is linear, 
	\item The running costs function $h$ is unimodal, convex right from its minimum, grows at least polynomially, $h'\left(x\right)>c>0$ for all $x<x_0$ for some $x_0$ and fulfil some more smoothness and integrability conditions (which in their full extent can be seen as Assumption 1 in \cite{yamazaki}).
\end{enumerate}
Under these conditions in \cite{yamazaki} it is shown that an optimal $(s,S)$ strategy exists (Theorem 1 therein), in Proposition 1 furthermore states that the value function can be expressed in terms of integral identities that comprise the running cost function, the scale functions and the L\'evy exponent of $X$, as well as the right inverse of the L\'evy exponent of $X$. The optimal pair of values $(s^*,S^*)$ is in Proposition 3 therein characterized as an optimizer of $\min_s\max_s \mathcal G(s,S)$, where $\mathcal G$ is also a function comprising all the objects that occur in the representation of the value function.

 In this section we use our solution technique to first prove existence of an optimal $(s,S)$ strategy in the long term average case under less restrictive assumption than the ones used in \cite{yamazaki} for the discounted case. Also we characterize the optimal value and the optimal boundaries using only the L\'evy triplet of $X$ and the root of the right inverse of its L\'evy exponent.

Namely we assume \begin{assumption}\label{beispielass}
	\begin{enumerate}
		\item $X$ is spectrally positive and all later occurring integrals exist.  
		\item For the payoff function $\gamma$ we have $\gamma\left(x\right)=Cx.$
		\item $h$ is positive and unimodal with unique minimum $a\in \R$, it only grows polynomially and we have 
		 $\lim_{x\rightarrow \infty} h\left(x\right)=\infty=\lim_{x\rightarrow -\infty} h\left(x\right)$.
	\end{enumerate}
\end{assumption}
Now we follow the steps laid out in Section \ref{sec:main}. \\
To tackle point 1. there, we first have to get a grip on $f$. 
Note that since $X$ is spectrally positive its descending ladder high process $H^{\downarrow}$ is just an exponentially with a positive rate $q>0$ killed drift, where $q=-\phi\left(0\right)$, $\phi$ being the right inverse of the Laplace exponent of $X$, see \cite{kypri}, Section 6.6.2. and Theorem 7.4. 
Hence the function $\hat{h}$ can be obtained via \begin{align*}
\hat{h}\left(x\right)=\int_{{ 0}}^\infty e^{-qt}h\left(t+x\right) \ dt
\end{align*}
Further,
\begin{align*}
A_H \gamma \left(x\right)=C\delta+C\int_0^\infty y \Pi_H\left(dy\right),
\end{align*}
where $\Pi_H$ is the L\'evy measure of the ladder height process $H$ and 
$\delta$ is the drift term of $H$, so $A_H\gamma$ does not depend on the exact characteristic of $X$. The L\'evy measure $\Pi_H$ can be expressed in terms of $q$ and the L\'evy measure $\Pi$ of $X$ via the formula \begin{align*}
\Pi_H\left(x,\infty\right)=e^{qx}\int_{x}^{\infty}e^{-qy}\Pi\left(y,\infty\right) \ dy,
\end{align*}
see \cite{kypri}, Section 6.6.2., and also Theorem 7.4, hence 
\begin{align*}
A_H \gamma \left(x\right)&=C\delta+\int_0^\infty y \Pi_H\left(dy\right)\\
&=C\delta+C\int_{0}^\infty \Pi_H\left(z,\infty\right) \ dz \\
&=C\delta+C\int_{0}^\infty e^{qz}\int_{z}^{\infty}e^{-qy}\Pi\left(y,\infty\right) \ dy \ dz \\
&=C\delta+C\int_{0}^\infty \int_{0}^{\infty}e^{-qy}\Pi\left(y+z,\infty\right)\ dz \ dy  \\
\end{align*}
This yields for all $x \in \R$ 
\begin{align*}
f\left(x\right)&=-A_H \gamma(x)-\hat h\left(x\right) \\&=-C\delta-\int_{0}^\infty \int_{0}^{\infty}e^{-qy}\Pi\left(y+z,\infty\right)\ dz \ dy-C\int_{{ 0}}^\infty e^{-qt}h\left(t+x\right) \ dt.
\end{align*}
Now since $\hat h$ is a Laplace transform and $A_H\gamma$ is constant, Assumption \ref{beispielass}, 3. yields that there is $\rho^*$ such that
\begin{align*}0&=\E_{\underline x_{\rho^*}}\left(\gamma\left(X_{\tau_{\overline x_{\rho^*}}}\right) -\int_0^{\tau_{\overline x_{\rho^*}}} \left(h\left(X_s\right)+{\rho^*} \right)\  ds\right)\\&=\E_{\underline x_{\rho^*}}\left(\int_{{ 0}}^{\hat \tau_{\overline x}} f_{\rho^*}\left(H_s\right)\  ds\right)
\\&=\E_{\underline x_{\rho^*}}\left(\int_{{ 0}}^{\hat \tau_{\overline x}} -\hat h\left(H_s\right)\  ds\right)+\left(A_H \gamma +\rho^*\right)\E_{\underline x_{\rho^*}}\left(\tau_{\overline x_{\rho^*}}\right).
\end{align*}

\subsection{Optimal harvesting}
Another field of application for our solution technique is optimal harvesting and forest management. This problem originates in the work of Martin Faustmann starting with his seminal paper \cite{faustmann1849berechnung} from 1849.
Until now, advancements and derivatives of this approach are used and usually called Faustmann's formula, see \cite{brazee} for an overview. In this branch of impulse control the underlying process models the growth of a forest, or more general: a natural resource, and impulse control theory is used to determine the optimal strategy to repeatedly harvest the wood. The question how to optimally exploit a natural resource whose dynamics involve randomness goes back several decades, see \cite{mayetalharvesting} for one of the earlier works. Nowadays there is a vast amount of literature present ranging from very applied to rather theoretical treatises (see, e.g., \cite{ bhattacharyya1988stumpage,willassen1998stochastic,A04,alvarez2006does,shackleton2010harvesting}).
While both modeling and solution approaches differ varying by the specific field of application, most of these works have in common that they describe the dynamics of the natural resource by a logistic diffusion.
\cite{alvarezshepp} provides a solution to the impulse control problem with an underlying logistic diffusion in the discounted case. Although in this fields the discounted pay-off functional is the most common choice, recently more and more works point out that on one hand it is difficult to determine the right discounting factor in practice and on the other hand the discounted model has the flaw to favor the present compared to the future and therefore might not be the right choice when one aims for sustainable solutions. The recent article \cite{alvarezherningltasingularharvesting} provides an example for the application of the long term average criterion to find a 'sustainable' harvesting strategy and a discussion of the model, see also \cite{christensen2019nonparametric}.
Here, we take a look at a typical Faustmann-type forest management problem as presented for instance in \cite{AL} or \cite{alvarez2006does}, but
we deviate from modeling the forest growth by a logistic diffusion. Instead we introduce downward jumps to the process since sudden events like storms, bushfire or diseases of the trees could abruptly destroy large quantities of the forest stand or make it worthless.
\subsubsection{Fixed restarting point}
In contrast to the other examples above where arbitrary downward shifts are admissible here first we stay in the classical Faustmann setting and use $B=\{0\}$, hence assume only one fixed restarting point. This is interpreted as a base level for the forest stand after harvesting. With the obvious alterations our solution technique still works in this case.  We assume $X$ to be spectrally negative, because trees don't just appear but grow continuously, set $h=0$ and consider a logistic-type gain function, for example $$
\gamma:\R\rightarrow \R; \ x \mapsto \frac{L}{1+e^{-sx}}
$$ for parameters $L, s>0$. The choice of $\gamma$ is motivated as follows: In the aforementioned works on optimal harvesting a diffusion is used to model the tree stand. Since, contrary to a diffusion model, we can't model different growth rates dependent on the current population with our L\'evy process directly, we interpret $\gamma\left(X_t\right)$ as the volume of wood in our forest present at time $t$.
Since there is no choice in the restarting point, no running costs and no upward jumps the procedure to find the optimal strategy breaks down to the following: \begin{enumerate}
	\item\label{itemeins} For arbitrary $\rho \in \R$ find the rightmost root $\bar x_\rho$ of $\E(X_1)\frac{d}{dx}\gamma -\rho $.
	\item Find $\rho^*$ such that $$\gamma\left(\bar x_{\rho^*}\right)-\gamma\left(0\right)-K-\rho^*\E_0\left(\tau_{\bar x_{\rho^*}}\right)=0.$$
\end{enumerate}
Note that in point \ref{itemeins}. $\E(X_1)$ occurs as the drift term of the ladder height process, since our choice of the ascending ladder times in Definition \ref{ladderheightdefinition} yields $\E(X_1)=\E(H_1)$ and due to $X$ being spectrally negative, the ladder height process does not jump either.

\subsubsection{Arbitrary downshifts allowed}
Now we consider the same setting as in the previous example with the only difference that we allow arbitrary downshifts, i.e., it is admissible to harvest just part of the timber.
Then the values $\underline x_{\rho^*}, \bar x_{\rho^*}, \rho^*$ can be found as follows: 
\begin{enumerate}
	\item For arbitrary $\rho \in \R$ find the two roots $\underline x_\rho, \bar x_\rho$ of $\E(X_1)\frac{d}{dx}\gamma -\rho $ (since in the example above $\gamma'$ is symmetric, we have $\bar x_\rho=-\underline x_\rho$).
	\item Find $\rho^*$ such that $$\gamma\left(\bar x_{\rho^*}\right)-\gamma\left(\underline x_{\rho^*}\right)-K-\rho^*\E_0\left(\tau_{\bar x_{\rho^*}}\right)=0.$$
\end{enumerate}

Of course these two examples mostly serve the purpose of easy examples to illustrate our findings nicely on a not too abstract level. Nevertheless, even this easy examples stress out some noteworthy observations: \begin{itemize}
	\item The only thing we have to know about the underlying L\'evy process (apart from the absence of upward jumps) is $\E\left(X_1\right)$. This opens the door to easy estimation and calculation procedures of the optimal boundaries. 
	\item More freedom in the choice of the restarting point of course yields a higher value for the control problem.
\end{itemize}

\appendix


\section{Toolbox}
This section serves as the collection of the tools we need. Since most of the results are well-known, we omit the proofs and just state the results. For a detailed treatise of renewal theory and the proofs of the lemmas originating in that field stated below, we refer to \cite{asmussen2008applied} and \cite{grimmetstirzaker}. 
\begin{lemma}[Wald's equation, discrete version]
	Let $\left(Y_i\right)_{i \in \N}$ be a sequence of independent random variables and $\tau$ a stopping time with respect to the filtration generated by the $Y_i$ such that one of the expressions $$\sum_{i=1}^{\infty}\E\left(|Y_i \mathds{1}_{\{\tau \geq i \}}|\right)$$ and $$\E\left(\sum_{i=1}^{\infty}|Y_i \mathds{1}_{\{\tau \geq i \}}|\right)$$ (and with Fubini-Tonelli both of them) is finite. Then
	\begin{enumerate}[(i)]\item $\E\left(\sum_{i=1}^\tau Y_i\right)=\E\left(\sum_{i=1}^\tau \E\left(Y_i\right)\right)$
		\item If the $Y_i$ are identically distributed, then $\E\left(\sum_{i=1}^\tau Y_i\right)=\E\left(\tau\right)\E\left(Y_1\right)$. 
	\end{enumerate}
\end{lemma}
\begin{remark}The assumptions of the previous lemma include the case that the $Y_i$ are positive. 
\end{remark}

\begin{lemma}[Renewal Reward Theorem]\label{rere}Assume $\left(Z_i,R_i\right)$ is a sequence of i.i.d. random variables, with $X_i>0$ a.s. for all $i \in \N$. Set $T_n:=\sum_{i\leq n}Z_i$ and $N\left(t\right):=\sup\{n\in \N | T_n \leq t \}$. Assume $\E\left(Z_1\right)< \infty$ and $\E\left(|R_1|\right)< \infty$. Then $$\frac{ \sum_{i=1}^{N\left(t\right)}R_i }{t}\stackrel{a.s.}{\rightarrow} \frac{\E\left(R_1\right)}{\E\left(Z_1\right)}  $$ and $$\frac{\E\left(\sum_{i=1}^{N\left(t\right)}R_i \right)}{t}\rightarrow \frac{\E\left(R_1\right)}{\E\left(Z_1\right)}. \ $$ 	\end{lemma}


\bibliography{LevyDrivenControlProblems}

\end{document}